\documentclass[11pt]{amsart}

\DeclareMathOperator{\Lie}{Lie}\DeclareMathOperator{\im}{im}

\DeclareMathOperator{\ad}{ad}

\DeclareMathOperator{\soc}{Soc}
\DeclareMathOperator{\Der}{Der}
\usepackage[usenames,dvipsnames]{color}

\usepackage{amsmath,amssymb,amsthm,graphics,amscd}
\usepackage{longtable,pifont}
\usepackage{cite,xypic}
\usepackage{color}
\usepackage{graphicx}
\usepackage{epsf}
\usepackage{fancyhdr,hyperref}
\usepackage{lscape}
\hypersetup{backref,
colorlinks=true,backref,citecolor=Purple}

  \renewenvironment{thebibliography}[1]{
    \begin{oldthebibliography}{#1}
      \setlength{\parskip}{0ex}
      \setlength{\itemsep}{0ex}
  }
  {
    \end{oldthebibliography}
  }
  
\setlongtables
\begin{document}

\newcommand{\cmark}{\ding{51}}%
\newcommand{\xmark}{\ding{55}}%

\newcounter{rownum}
\setcounter{rownum}{0}
\newcommand{\ab}{\addtocounter{rownum}{1}\arabic{rownum}}

\newcommand{\x}{$\times$}
\newcommand{\bb}{\mathbf}

\newcommand{\Ind}{\mathrm{Ind}}
\newcommand{\Char}{\mathrm{char}}
\newcommand{\hra}{\hookrightarrow}
\newtheorem{lemma}{Lemma}[section]
\newtheorem{theorem}[lemma]{Theorem}
\newtheorem*{TA}{Theorem A}
\newtheorem*{TB}{Theorem B}
\newtheorem*{TC}{Theorem C}
\newtheorem*{CorC}{Corollary C}
\newtheorem*{TD}{Theorem D}
\newtheorem*{TE}{Theorem E}
\newtheorem*{PF}{Proposition E}
\newtheorem*{C3}{Corollary 3}
\newtheorem*{T4}{Theorem 4}
\newtheorem*{C5}{Corollary 5}
\newtheorem*{C6}{Corollary 6}
\newtheorem*{C7}{Corollary 7}
\newtheorem*{C8}{Corollary 8}
\newtheorem*{claim}{Claim}
\newtheorem{cor}[lemma]{Corollary}
\newtheorem{conjecture}[lemma]{Conjecture}
\newtheorem{prop}[lemma]{Proposition}
\newtheorem{question}[lemma]{Question}
\theoremstyle{definition}
\newtheorem{example}[lemma]{Example}
\newtheorem{examples}[lemma]{Examples}
\newtheorem{algorithm}[lemma]{Algorithm}
\newtheorem*{algorithm*}{Algorithm}
\theoremstyle{remark}
\newtheorem{remark}[lemma]{Remark}
\newtheorem{remarks}[lemma]{Remarks}
\newtheorem{obs}[lemma]{Observation}
\theoremstyle{definition}
\newtheorem{defn}[lemma]{Definition}

  \def\hal{\unskip\nobreak\hfil\penalty50\hskip10pt\hbox{}\nobreak
  \hfill\vrule height 5pt width 6pt depth 1pt\par\vskip 2mm}

\renewcommand{\labelenumi}{(\roman{enumi})}
\newcommand{\Hom}{\mathrm{Hom}}
\newcommand{\Int}{\mathrm{int}}
\newcommand{\Ext}{\mathrm{Ext}}
\newcommand{\opH}{\mathrm{H}}
\newcommand{\D}{\mathcal{D}}
\newcommand{\SO}{\mathrm{SO}}
\newcommand{\Sp}{\mathrm{Sp}}
\newcommand{\SL}{\mathrm{SL}}
\newcommand{\GL}{\mathrm{GL}}
\newcommand{\OO}{\mathcal{O}}
\newcommand{\Y}{\mathbf{Y}}
\newcommand{\FF}{\mathcal{F}}
\newcommand{\X}{\mathbf{X}}
\newcommand{\diag}{\mathrm{diag}}
\newcommand{\End}{\mathrm{End}}
\newcommand{\tr}{\mathrm{tr}}
\newcommand{\Stab}{\mathrm{Stab}}
\newcommand{\red}{\mathrm{red}}
\newcommand{\Aut}{\mathrm{Aut}}
\renewcommand{\H}{\mathcal{H}}
\renewcommand{\u}{\mathfrak{u}}
\newcommand{\Ad}{\mathrm{Ad}}
\newcommand{\N}{\mathcal{N}}
\newcommand{\Z}{\mathbb{Z}}
\newcommand{\la}{\langle}\newcommand{\ra}{\rangle}
\newcommand{\gl}{\mathfrak{gl}}
\newcommand{\g}{\mathfrak{g}}
\newcommand{\F}{\mathbb{F}}
\newcommand{\m}{\mathfrak{m}}
\renewcommand{\b}{\mathfrak{b}}
\newcommand{\p}{\mathfrak{p}}
\newcommand{\q}{\mathfrak{q}}
\renewcommand{\l}{\mathfrak{l}}
\newcommand{\del}{\partial}
\newcommand{\h}{\mathfrak{h}}
\renewcommand{\t}{\mathfrak{t}}
\renewcommand{\k}{\mathfrak{k}}
\newcommand{\Gm}{\mathbb{G}_m}
\renewcommand{\c}{\mathfrak{c}}
\renewcommand{\r}{\mathfrak{r}}
\newcommand{\n}{\mathfrak{n}}
\newcommand{\s}{\mathfrak{s}}
\newcommand{\Q}{\mathbb{Q}}
\newcommand{\C}{\mathbb{C}}
\newcommand{\z}{\mathfrak{z}}
\newcommand{\pso}{\mathfrak{pso}}
\newcommand{\so}{\mathfrak{so}}
\renewcommand{\sl}{\mathfrak{sl}}
\newcommand{\psl}{\mathfrak{psl}}
\renewcommand{\sp}{\mathfrak{sp}}
\newcommand{\Ga}{\mathbb{G}_a}

\newenvironment{changemargin}[1]{%
  \begin{list}{}{%
    \setlength{\topsep}{0pt}%
    \setlength{\topmargin}{#1}%
    \setlength{\listparindent}{\parindent}%
    \setlength{\itemindent}{\parindent}%
    \setlength{\parsep}{\parskip}%
  }%
  \item[]}{\end{list}}

\parindent=0pt
\addtolength{\parskip}{0.5\baselineskip}

\title{The Jacobson--Morozov theorem and complete reducibility of Lie subalgebras}
\author{David I. Stewart}
\address{University of Newcastle, UK} \email{dis20@cantab.net {\text{\rm(Stewart)}}}

\author{Adam R. Thomas*} \thanks{*Supported by an LMS 150th Anniversary Postdoctoral Mobility Grant 2014-2015 Award}
\address{School of Mathematics, University of Bristol, Bristol, BS8 1TW, UK, and the Heilbronn Institute for Mathematical Research, Bristol, UK.} 
\email{adamthomas22@gmail.com {\text{\rm{(Thomas)}}}}
\pagestyle{plain}
\begin{abstract}In this paper we determine the precise extent to which the classical $\sl_2$-theory of complex semisimple finite-dimensional Lie algebras due to Jacobson--Morozov and Kostant can be extended to positive characteristic. This builds on work of Pommerening and improves significantly upon previous attempts due to Springer--Steinberg and Carter/Spaltenstein. Our main advance arises by investigating quite fully the extent to which subalgebras of the Lie algebras of semisimple algebraic groups over algebraically closed fields $k$ are $G$-completely reducible, a notion essentially due to Serre. For example, if $G$ is exceptional and $\mathrm{char}\ k=p\geq 5$, we classify the triples $(\h,\g,p)$ such that there exists a non-$G$-completely reducible subalgebra of $\g=\Lie(G)$ isomorphic to $\h$. We do this also under the restriction that $\h$ be a $p$-subalgebra of $\g$. We find that the notion of subalgebras being $G$-completely reducible effectively characterises when it is possible to find bijections between the conjugacy classes of $\sl_2$-subalgebras and nilpotent orbits and it is this which allows us to prove our main theorems. 

For absolute completeness, we also show that there is essentially only one occasion in which a nilpotent element cannot be extended to an $\sl_2$-triple when $p\geq 3$: this happens for the exceptional orbit in $G_2$ when $p=3$.
\end{abstract}
\maketitle

\section{Introduction}
The Jacobson--Morozov theorem is a  fundamental result in the theory of complex semisimple Lie algebras, due originally to Morozov, but with a corrected proof by Jacobson. One way to state it is to say that for any complex semisimple Lie algebra $\g=\Lie(G)$, there is a surjective map \[\{\text{conjugacy classes of }\sl_2\text{-triples}\}\longrightarrow\{\text{nilpotent orbits in $\g$}\}\tag{*},\] induced by sending the $\sl_2$-triple $(e,h,f)$ to the nilpotent element $e$. That is, any nilpotent element $e$ can be embedded into some $\sl_2$-triple. In \cite{Kos59}, Kostant showed that this can be done uniquely up to conjugacy by the centraliser $G_e$ of $e$; i.e. that the map (*) is actually a bijection. Much work has been done on extending this important result to the modular case, that is where $\g=\Lie(G)$ for $G$ a reductive algebraic group over an algebraically closed field $k$ of characteristic $p>0$. We mention some critical contributions. In \cite{Pom80}, Pommerening showed that under the mild restriction that $p$ is a good prime, one can always find an $\sl_2$-subalgebra containing a given nilpotent element, but this may not be unique; in other words, the map (*) is still surjective, but not necessarily injective.  
If $h(G)$ denotes the Coxeter number\footnote{In case the root system of $G$ is trivial, we define $h(G)$ to be $0$.} of $G$, then in \cite{SpSt70} Springer and Steinberg prove that the uniqueness holds whenever $p\geq 4h(G)-1$ and in his book \cite{Car93}, Carter uses an argument due to Spaltenstein to establish the result under the weaker condition $p>3h(G)-3$; both proofs go essentially  by an exponentiation argument. One major purpose of this article is to finish this project and improve these bounds on the characteristic optimally, thus to state precisely when the bijection (*) holds. 

\begin{theorem}\label{jmthm} Let $G$ be a connected reductive group  over an algebraically closed field $k$ of characteristic $p>2$\footnote{Note that if $p=2$ then $\sl_2$ is non-simple and the question of finding subalgebras containing given nilpotent elements becomes murky since one might consider it proper to consider the non-simple non-isomorphic subalgebras $\mathfrak{pgl}_2$ in addition.} and let $\g$ be its Lie algebra. Then (*) is a bijection if and only if $p>h(G)$.\footnote{Interestingly, this theorem gives an optimal answer as to when the secondary demands of \cite[Hypothesis 4.2.5]{DeB02} are met; however it is known to the authors that   \cite[Hypothesis 4.2.4]{DeB02} on which it is dependent holds for every nilpotent orbit only under strictly worse conditions.}\end{theorem}

In fact, we will do even more than this, also determining when there is a bijection 
\[\{\text{conjugacy classes of }\sl_2\text{-subalgebras}\}\longrightarrow\{\text{nilpotent orbits in $\g$}\}\tag{**},\]
and when a bijection exists, we will be able to realise it in a natural way. The equivalence of bijections (*) and (**) is easily seen in large enough characteristics by exponentiation, but there are quite a few characteristics where there exists a bijection (**), but not (*). To state our result, we define for any reductive group the number $b(G)$ as the largest prime $p$ such that the Dynkin diagram of $G$ contains a subdiagram of type $A_{p-1}$ or $p$ is a bad prime for $G$. Alternatively, $b(G)$ is the largest prime which is not very good for some Levi subgroup of $G$. If $G$ is classical of type $(A_n,B_n,C_n,D_n)$ we have $b(G)$ is the largest prime which is no larger than $(n+1,n,n,n)$ and if $G$ is exceptional of type $(G_2,F_4,E_6,E_7,E_8)$ then $b(G)=(3,3,5,7,7)$. If $G$ is reductive then $b(G)$ is the maximum over all simple factors and is $0$ in case the root system of $G$ is trivial.

\begin{theorem}\label{sl2subalgebrabij} Let $G$ be a connected reductive group  over an algebraically closed field $k$ of characteristic $p>2$ and let $\g$ be its Lie algebra. Then the number of conjugacy classes of $\sl_2$-subalgebras and nilpotent orbits is the same if and only if $p>b(G)$. Moreover, when $p>b(G)$, there is a natural bijection (**) realised by sending an $\sl_2$-subalgebra $\h$ to the nilpotent orbit of largest dimension that intersects $\h$ non-trivially.\end{theorem}
(To emphasise our improvement, \cite[Thm.~5.5.11]{Car93} gives the existence of such a bijection for $E_8$ when $p>87$, whereas we require just $p>7$.) 

For many applications, the Kempf--Rousseau theory of optimal cocharacters (whose consequences were worked out in \cite{Pre03}) is a sufficient replacement for much of the $\sl_2$-theory one would typically employ when working over $\C$---indeed, this paper uses cocharacter theory quite extensively. But it should not be a surprise that the unique smallest simple Lie algebra over $k$ should continue to play a role in modular Lie theory. We are aware of at least one example where our results are likely to be used: on considering a maximal subgroup $H$ of a finite group of Lie type $G(q)$, one frequently discovers the existence of a unique $3$-dimensional submodule on the adjoint module that must correspond to an $\sl_2$-subalgebra of $\g$. Then Theorem \ref{sl2subalgebrabij} promises certain useful properties of this subalgebra which can be exploited to show that $H$ lives in a positive-dimensional subgroup of the ambient algebraic group $G$; typically this implies it is not maximal.

The question of the existence of bijections (*) and (**) turns out to be intimately connected to J.-P.~Serre's notion of $G$-complete reducibility \cite{Ser05}. Say a subgroup $H$ of $G$ is $G$-completely reducible if whenever $H$ is contained in a parabolic subgroup $P$ of $G$, then $H$ is contained in a Levi subgroup $L$ of $P$. The notion is inspired by a general philosophy of Serre, Tits and others to generalise concepts of representation theory by replacing homomorphisms of groups $H\to\GL(V)$ with homomorphisms of groups $H\to G$, where $G$ is any reductive algebraic group. Indeed, when $G=\GL(V)$, using  the description of the parabolic subgroups and Levi subgroups of $G$ as  stabilisers of flags of subspaces of $V$, the idea that a subgroup $H$ is $G$-completely reducibly recovers the usual idea of $H$ acting completely reducibly on a representation $V$. There is a remarkably widespread web of connections between $G$-complete reducibility and other areas of mathematics, such as geometric invariant theory, the theory of buildings and the subgroup structure of algebraic groups, amongst other things. In our proofs of Theorem \ref{jmthm} and \ref{sl2subalgebrabij} we will find yet another connection with Serre's notion, this time with the study of modular Lie algebras. 

The natural extension of Serre's idea to Lie algebras is due to McNinch, \cite{McN07} and is developed further in \cite{BMRT13}. We say a subalgebra $\h$ of $\g$ is $G$-completely reducible (or $G$-cr) if whenever $\h$ is contained in a parabolic subalgebra $\p$ of $\g$, then $\h$ is in a Levi subalgebra of that parabolic. (Recall that a parabolic subalgebra is by definition $\Lie(P)$ for $P$ a parabolic subgroup and a Levi subalgebra is $\Lie(L)$ for $L$ a Levi subgroup of $P$.) We will establish the following result, crucial for our proof of Theorem \ref{jmthm}.
\begin{theorem}\label{ssgcrcorollary} Let $G$ be a connected reductive algebraic group over an algebraically closed field $k$ of characteristic $p>2$. Then all semisimple subalgebras of $\g$ are $G$-completely reducible if and only if $p>h(G)$.\end{theorem}
The proof of Theorem \ref{ssgcrcorollary} reduces easily to the case where $G$ is simple. Then work of S.~Herpel and the first author in \cite{HS14} on complete reducibility of representations of semisimple Lie algebras can be adapted to prove the theorem when $G$ is classical. The bulk of the work involved is showing the result when $G$ is exceptional. Let then $G$ be an exceptional algebraic group. At least thanks to \cite{HSMax} and some work of A.~Premet together with the first author, the isomorphism classes of semisimple Lie subalgebras of the exceptional Lie algebras are known in all good characteristics.\footnote{Up to isomorphism, one gets only Lie algebras coming from algebraic groups and the first Witt algebra $W(1;\underline{1})$ of dimension $p$, together with some semisimple subalgebras which are not the direct sum of simple Lie algebras, existing only when $p=5$ or $7$ and $\g$ is of type $E_7$ or $E_8$. An example of such a Lie algebra is the semidirect product $\sl_2\otimes (k[X]/\la X^p\ra) + 1\otimes W(1;\underline{1})$ where the factor $W(1;\underline{1})$ commutes with the $\sl_2$ factor but acts by derivations on the truncated polynomial ring $k[X]/\la X^p\ra$.} Our following theorem gives, for $p\geq 5$ (in particular, for $p$ a good prime) a full description of when a simple subalgebra in one of those known isomorphism classes can be non-$G$-cr.
\begin{theorem}\label{maintheorem}Suppose $\h$ is a simple subalgebra of $\g=\Lie(G)$ for $G$ a simple algebraic group of exceptional type over an algebraically closed field $k$ of characteristic $p\geq 5$.  Then either $\h$ is $G$-cr or one of the following holds:
\begin{enumerate}\item $\h$ is of type $A_1$ and $p<h(G)$;
\item $\h$ is of type $W(1;\underline{1})$, $p=7$ and $G$ is of type $F_4$; or $p=5$ or $7$ and $G$ is of type $E_6$, $E_7$ or $E_8$;
\item Up to isomorphism we have $(G,\h,p)=(E_7,G_2,7)$, $(E_8,B_2,5)$ or $(E_8,G_2,7)$.
\end{enumerate}
Moreover, for each exception $(G,\h,p)$ above, there exists a non-$G$-cr subalgebra of the stated type.
\end{theorem}

Since we consider Lie algebras $\g=\Lie(G)$ for $G$ an algebraic group, $\g$ inherits a $[p]$-map arising from the Frobenius morphism on the group. Then a subalgebra $\h$ of $\g$ is a \emph{$p$-subalgebra} if and only if it is closed under the $[p]$-map. Asking when $p$-subalgebras are $G$-cr gives a slightly different answer, with an important connection to the existence of the bijection (**). 
\begin{theorem}\label{psubalgebrathm}Suppose $\h$ is a simple $p$-subalgebra of $\g=\Lie(G)$ for $G$ a simple algebraic group of exceptional type over an algebraically closed field $k$ of characteristic $p\geq 5$. Then either $\h$ is $G$-cr or one of the following holds:
\begin{enumerate}\item $\h$ is of type $A_1$ and $p \leq b(G)$;
\item $\h$ is of type $W(1;\underline{1})$, $p=7$ and $G$ is of type $F_4$; or $p=5$ or $7$ and $G$ is of type $E_6$, $E_7$ or $E_8$;
\item Up to isomorphism we have $(G,\h,p)=(E_7,G_2,7)$, $(E_8,B_2,5)$ or $(E_8,G_2,7)$.
\end{enumerate}
Moreover, for each exception, there exists a non-$G$-cr $p$-subalgebra of the stated type.
\end{theorem}


To appreciate fully the connection of Theorems \ref{ssgcrcorollary} and \ref{maintheorem} with Theorem \ref{jmthm} we will see that the failure of the uniqueness part of the Jacobson--Morozov theorem to hold in characteristics less than or equal to the Coxeter number $h(G)$ comes exactly from the failure of some subalgebras isomorphic to $\sl_2$ to be $G$-cr. (And this is precisely how we construct examples of extra conjugacy classes of $\sl_2$ subalgebras when $p\leq h(G)$.) Moreover, so long as $G$ contains neither a factor of type $G_2$ when $p=3$ nor a factor of type $A_{p-1}$, then the bijection (**) in Theorem \ref{sl2subalgebrabij} exists precisely when there is an equivalence \[H \text{ is $G$-completely reducible }\iff H\text{ is reductive}\] for all connected reductive subgroups $H$ of $G$. 


Another result concerns a connection between Seitz's idea of subgroups of type $A_1$ being good with the study of modular Lie algebras. Recall from \cite{Sei00} that a closed subgroup $H$ of type $A_1$ of an algebraic group $G$ is \emph{good} if it has weights no bigger than $2p-2$ on the adjoint module. Again, this idea forms part of the philosophy of generalising concepts of representation theory from $\GL(V)$ to other reductive groups. This time, Seitz's notion gives us the correct generalisation of the notion of a restricted representation of $H:=\SL_2$: If $H$ acts with weights less than $p$ on $V$, then it gives a good $A_1$-subgroup of $\GL(V)$, since $H$ will have weights no more than $2p-2$ on $\gl(V)|_H\cong V\otimes V^*$. In {\it ibid.} Seitz proves in particular that all unipotent elements of order $p$ have a good $A_1$-overgroup and that any two such overgroups are conjugate; this itself connects to questions raised in Serre's fundamental paper \cite{Ser05} by providing a solution to finding overgroups of unipotent elements which are so-called `saturated'. Our result is as follows.

\begin{theorem}\label{gooda1s}Let $G$ be a connected reductive algebraic group over an algebraically closed field of characteristic $p>2$. Then every $\sl_2$ subalgebra of $G$ is $\Lie(H)$ for $H$ a good $A_1$ if and only if $p>h(G)$.\end{theorem}

Lastly, for completeness we have checked the following, improving the Jacobson--Morozov theorem itself optimally, using the classification of nilpotent orbits in characteristic $p\geq 3$.
\begin{theorem}\label{jm}Let $G$ be a simple algebraic group over an algebraically closed field of characteristic $p\geq 3$. Then any nilpotent element $e \in \g = \Lie(G)$ belonging to the orbit $\mathcal O$ can be extended to an $\sl_2$-triple if and only if $(G,p,\mathcal O)\neq (G_2,3,\tilde A_1^{(3)})$.\end{theorem} 

\subsection*{Acknowledgements} We thank Sasha Premet for some discussion and the referee for helpful suggestions for improvement.

\section{Preliminaries} \label{sec:prelims}
\subsection{Notation}
In the following $G$ will be a reductive algebraic group over an algebraically closed field $k$ of characteristic $p>2$, and $\g$ will be its Lie algebra. 

Throughout the paper we use the terms classical and exceptional when referring to both simple algebraic groups and simple Lie algebras. When we say a simple Lie algebra is classical (or of classical type) we mean that it is of type $A$--$G$. However, for a simple algebraic group, we use the term classical to mean of type $A$--$D$, and exceptional otherwise. 

All notation unless otherwise mentioned will be consistent with \cite{Jan03}. In particular, all our reductive groups are assumed to be connected. The root system $R$ contains a simple system $S$ whose elements will be denoted $\alpha_i$, with corresponding fundamental dominant weight $\varpi_i$. We shall denote roots in $R$ by their coefficients in $S$ labelled consistently with Bourbaki. For a dominant weight $\lambda = a_1 \varpi_1 + a_2 \varpi_2 + \cdots + a_n \varpi_n$ we write $\lambda = (a_1, a_2, \dots, a_n)$ and write $L(\lambda) = L(a_1, a_2, \dots, a_n)$ to denote the irreducible module of highest weight $\lambda$. Given modules $M_1, \ldots, M_k$, the notation $V = M_1 | \dots | M_k$ denotes a module with a socle series as follows: $\mathrm{Soc}(V) \cong M_k$ and $\mathrm{Soc}(V / \sum_{j \geq i}M_j) \cong M_{i-1}$ for $k \geq i > 1$. We write $M_1 + M_2$ for $M_1\oplus M_2$. We also write $T(\lambda)$ for a tilting module of high weight $\lambda$ for an algebraic group $G$. In small cases, the structure of these is easy to write down. For example, when $V(\lambda)\cong L(\lambda)|L(\mu)$, we have $T(\lambda)\cong L(\mu)|L(\lambda)|L(\mu)$. The module $V_{E_6}(\varpi_1)$ will be denoted $V_{27}$ and the module $V_{E_7}(\varpi_7)$ denoted $V_{56}$.  

When $G$ is simple and simply-connected, we choose root vectors in $\g$ for a torus $T\subseteq G$ and a basis for $\t = \Lie(T)$ coming from a basis of subalgebras isomorphic to $\sl_2$ corresponding to each of the simple roots. We write these elements as $\{ e_\alpha : \alpha \in R\}$ and $\{ h_\alpha : \alpha \in S\}$ respectively. As $\g=\Lie(G)$, we have that $\g$ inherits a $[p]$-map $x\mapsto x^{[p]}$, making it a restricted Lie algebra; see \cite[I.7.10]{Jan03}.

Recall also the first Witt algebra $W(1;\underline{1}):=\Der_k(k[X]/X^p)$, henceforth denoted $W_1$. The Lie algebra $W_1$ is $p$-dimensional with basis $\{\del,X\del,\dots,X^{p-1}\del\}$ and commutator formula $[X^i\del,X^j\del]=(j-i)X^{i+j-1}\del$. In \S\ref{sec:w1s} we use a little of the representation theory of $W_1$. All that we need is contained in \cite{BNW09} for example. In particular, $\Der_k(k[X]/X^p)$ is a module with structure $S|k$ where $S$ is an irreducible module of dimension $p-1$.

\subsection{Parabolic subalgebras}\label{sec:parabs}
Let $P = LQ$ be a standard parabolic subgroup of an exceptional algebraic group $G$ with unipotent radical $Q$ and Levi complement $L$, corresponding to a subset $J$ of $S$. In particular, letting $R_J = R \cap \Z J$, we have $P = \langle U_\alpha, \, T \, | \, \alpha \in R^+ \cup R_J \rangle$. In this section we discuss the structure of $Q$ in terms of the action of $L$. Forgetting the Lie algebra structure on $\q:=\Lie(Q)$, we obtain a module for $\l:=\Lie(L)$. We will see that if $\h$ is a subalgebra of $\l$ such that $\q$ has no $\h$-composition factors $V$ with non-trivial first cohomology (which we will recall shortly) then all complements to $\q$ in the semidirect product of Lie algebras $\h+\q$ are conjugate to $\h$ by elements of $Q$, hence all are $G$-cr.

The unipotent radical has (by \cite{ABS90}) a well-known central filtration $Q=Q_1\geq Q_2\geq \dots$ with successive factors $Q_i/Q_{i+1}$ isomorphic to the direct product of all root groups corresponding to the set $\Phi_i$ of roots of level $i$, where the level of a root $\alpha = \sum_{i \in S} c_{i} \alpha_i$ is $\sum_{i \in S \setminus J} c_{i}$, via the multiplication map $\pi:\Ga\times\dots\times\Ga\to G$; $(t_1,\dots,t_n)\mapsto\prod_{\alpha\in\Phi_i}x_\alpha(t_i)$. The filtration $\{Q_i\}$ is $L$-stable and the quotients have the structure of $L$-modules. That is, they are $L$-equivariantly isomorphic to the $L$-module $\Lie(Q_i/Q_{i+1})=\Lie(Q_i)/\Lie(Q_{i+1})$, as is verified in \cite{ABS90}. Moreover it is straightforward to compute the $L$-composition factors of each subquotient; see  \cite[Lem.~3.1]{LS96}. One observes all of the high weights are restricted when $p$ is good for $G$ (and for $p=5$ when $G=E_8$). We may therefore immediately conclude by differentiating the $L$-modules concerned that the same statement is true of the $\l$-composition factors of the $\l$-module $\Lie(Q_i)/\Lie(Q_{i+1})$. The following lemma records this.

\begin{lemma} \label{compsofparabs}
Let $\g$ be the Lie algebra of a simple exceptional algebraic group $G$ in good characteristic (or $p=5$ when $G=E_8$) and let $\p = \l+\q$ be a parabolic subalgebra of $\g$. The $\l$-composition factors within $\q$ have the structure of high weight modules for $\l$. If $\l_0=\Lie(L_0)$ for $L_0$ a simple factor of $L$, then the possible high weights $\lambda$ of non-trivial $\l_0$-composition factors are as follows: 
\begin{enumerate}
 \item $\l_0 = A_n$: $\lambda = 2 \varpi_1$, $2\varpi_n$, $3 \varpi_1$, $\varpi_j$ or $\varpi_{n+1-j}$ $(j =1, 2, 3$) (note that $2 \varpi_1$, $2 \varpi_n$ only occur if $\g = F_4$ and $n \leq 2$ and $3 \varpi_1$ only if $\g = G_2$ and $n=1$); 
 \item $\l_0 = B_n$ or $C_n$ ($n=2$ or $3$, $\g = F_4$): $\lambda = \varpi_1, \varpi_2$ or $\varpi_3$;
 \item $\l_0 = D_n$: $\lambda = \varpi_1, \varpi_{n-1}$ or $\varpi_n$;
 \item $\l_0 = E_6$: $\lambda = \varpi_1$ or $\varpi_6$; 
 \item $\l_0 = E_7$: $\lambda = \varpi_7$. 
\end{enumerate}
 
\end{lemma}

We therefore find the following restrictions on the dimensions of $\l$-composition factors of $\q$ (hence also on the $\h$-composition factors of $\q$).

\begin{cor} \label{dimofparabs}
With the hypothesis of the lemma, let $V$ be an $\l$-composition factor of $\q$. Then $\dim V \leq 64$. 
\end{cor}

\begin{proof}
This follows from the lemma if $\l'$ is simple. Moreover, if $\g \not = E_8$ then the number of positive roots is at most 56 and the result follows. So suppose $\g = E_8$. The product of the dimensions of the possible simple factors is at most 64 in all cases, except for $\l'$ of type $A_1 A_6$ for which a module $L(1) \otimes L(\varpi_3)$ has dimension $2 \times 35 = 70$. However, an easy calculation shows the actual $\l'$-composition factors are $L(1) \otimes L(\varpi_2)$, $L(\varpi_4)$, $L(1) \otimes L(\varpi_6)$ and $L(\varpi_1)$. Hence the largest dimension of any $\l'$-composition factor is $42$. 
\end{proof}

We recall a concrete description of the $1$-cohomology of Lie algebras; see \cite[\S7.4]{Wei94}. Let $\h$ be a Lie algebra and $V$ an $\h$-module. A $1$-cocycle is a map $\varphi:\h\to V$ such that \begin{equation}\varphi([x,y])=x(\varphi(y))-y(\varphi(x)).\label{cocyclecondition}\end{equation} Let $Z^1(\h,V)$ denote the set of $1$-cocycles. For $v\in V$ the map $\h\to V:x\mapsto x(v)$ is a $1$-cocycle called a $1$-coboundary; denote these by $B^1(\h,V)$. Two $1$-cocycles are equivalent if they differ by a $1$-coboundary; explicitly $\varphi\sim\psi$ if there is some $v\in V$ such that $\varphi(x)=\psi(x)+x(v)$ for all $x\in\h$. In this case we say $\varphi$ and $\psi$ are \emph{conjugate} by $v$. One then has $\opH^1(\h,V)=Z^1(\h,V)/B^1(\h,V)$. 

Note that $V$ can be considered as a Lie algebra with trivial bracket. Then one may form the semidirect product $\h+V$. A complement to $V$ in the semidirect product $\h+V$ is a subalgebra $\h'$ such that $\h'\cap V=0$ and $\h'+V=\h+V$. Just as for groups, one has a vector space isomorphism \[Z^1(\h,V)\longleftrightarrow \left\{\text{complements to }V\text{ in }\h+V\right\},\] by $\varphi\mapsto \{x+\varphi(x):x\in\h\}$.

We wish to realise the conjugacy action of $V$ on $\h+V$ in terms of a group action. Suppose $\dim V=n$. For our purposes it will do no harm to identify $\h$ with its image in $\gl_n$.  Furthermore, it will be convenient to embed $V$ into $\gl_{n+1}$ as strictly upper-triangular matrices with non-zero entries only in the last column, {\it viz.}
\[V:=\left(\begin{matrix}0&&&&*\\
&0&&&*\\
&&\ddots&&\vdots\\
&&&0&*\\
&&&&0\end{matrix}\right).\]
Then the action of $\GL_n$ on $V$ is realised as conjugation of the block diagonal embedding of $\GL_n$ in $\GL_{n+1}$ via $x\mapsto \diag(x,1)$. Clearly adding the identity endomorphism of $V$ commutes with the action of $\GL_n$. Hence  the group $Q:=1+V$ is $\GL_n$-equivariant to its Lie algebra $V$. Now suppose $\h'$ is a complement to $\h$ in $\h+V$ given as $\h'=\{x+\varphi(x):x\in\h\}\subset \gl_{n+1}$. If $q=1+v\in Q$ with $v\in V$ we have \[(x+\varphi(x))^q=x^q+\varphi(x)=(1-v)x(1+v)+\varphi(x).\]
Then since $x$ normalises $V$ and any two endomorphisms $v_1,v_2\in V\subseteq\End(V)$ satisfy $v_1v_2=0$, we see easily that $(x+\varphi(x))^q=x+[x,v]+\varphi(x)=x+x(v)+\varphi(x)$, showing that two cocycles in $Z^1(\h,V)$ are equivalent if and only if they are conjugate under the action $\varphi^q(x)=x(v)+\varphi(x)$, where $q=1+v$.

We have proven the following proposition:

\begin{prop}\label{h1moduleiszero}Let $\h\subseteq\gl(V)$ be a subalgebra with $\dim V=n$. Then realised as a subalgebra of $\gl_{n+1}$ as above, all complements to $V$ in $\h+V$ are conjugate to $\h$ via elements of $1+V$ if and only if $\opH^1(\h,V)=0$.\end{prop}

%
%

We also require the following crucial result, generalising the above proposition to the case where $\q$ is non-abelian. This is one of  our main tools in proving Theorem \ref{maintheorem}:

\begin{prop}\label{h1ofqis0}Suppose $P=LQ$ is a parabolic subgroup of a reductive algebraic group $G$ with unipotent radical $Q$ and Levi factor $L$. Let $\h$ be a subalgebra of $\Lie(L)$ and suppose that $\opH^1(\h,\Lie(Q_i/Q_{i+1}))=0$ for each $i\geq 1$. Then all complements to $\q$ in $\h+\q$ are $Q$-conjugate to $\h$.\end{prop}
\begin{proof}Let $Q=Q_1\geq Q_2\geq \dots $ denote the filtration of $Q$ by $L$-modules described at the beginning of \S2.2 and let $\q_i:=\Lie(Q_i)$. We prove inductively that all complements to $\q/\q_i$ in $\h+\q/\q_i$ are $Q/Q_i$-conjugate to $\h$. If $i=1$ this is trivial, so assume all complements to $\q$ in $\h+\q/\q_{i-1}$ are $Q/Q_{i-1}$-conjugate to $\h$. Take a complement $\h'$ to $\q/\q_i$ in $\h+\q/\q_{i}$. Then by the inductive hypothesis, we may replace $\h'$ by a $Q/Q_i$-conjugate so that $\h+\q_{i-1}/\q_i=(\h+\q_{i-1})/\q_i=(\h'+\q_{i-1})/\q_i$. Thus $\h'$ is a complement to $\q_{i-1}/\q_i$ in the semidirect product $\h+\q_{i-1}/\q_i$. We have $\h'=x+\gamma(x)$ for some cocycle $\gamma:\h\to\q_{i-1}/\q_i$. But since $\opH^1(\h,\q_{i-1}/\q_i)=0$, we may write $\gamma(x)=[x,v]$ for some $v\in\q$. Now $\gamma|_{\c_\h(\q)}$ is identically zero. Thus $\gamma$ factorises through $\h\to\gl(\q_{i-1}/\q_{i})$, so it suffices to consider the image of $\h$ in $\gl(\q_{i-1}/\q_{i})$.  
Now by Proposition \ref{h1moduleiszero} the image of $\h'$ is $Q_{i-1}/Q_i$-conjugate to $\h$. It follows that $\h'$ is $Q_{i-1}/Q_i$-conjugate to $\h$.
\end{proof}

\subsection{Cohomology for $G$, $G_1$ and $\g$}\label{sec:h1s}

In this section, $G$ will be simple and simply-connected. In the subsequent analysis, it will be necessary to know the ordinary Lie algebra cohomology groups $\opH^1(\g,V)$ for $V$ a simple restricted $\g$-module of small dimension. We will first reduce our considerations to the restricted Lie algebra cohomology. In this section, $G_1$ denotes the first Frobenius kernel of $G$---we warn the reader that we sometimes use subscripts to index groups, but it will be clear from the context our intent. We may identify restricted cohomology for $\g$ with the usual Hochschild cohomology for $G_1$. Recall the exact sequence \cite[I.9.19(1)]{Jan03}:\begin{align}\label{seq}0&\to \opH^1(G_1,V)\to \opH^1(\g,V)\to\Hom^s(\g,V^\g)\nonumber\\
&\to \opH^2(G_1,V)\to \opH^2(\g,V)\to \Hom^s(\g,\opH^1(\g,V)).\end{align}

From this it follows that if $V^\g=0$ (i.e. $\g$ has no fixed points on $V$) then $\opH^1(G_1,V)\to \opH^1(\g,V)$ is an isomorphism. This happens particularly in the case that $V$ is simple and non-trivial. The main result of this section is Proposition \ref{h1sless64} below, which gives every instance of a simple restricted $G$-module $V$ of dimension at most $64$ such that the groups $\opH^1(G_1,V)\cong\opH^1(\g,V)$ are non-trivial. In order to prove this proposition, we will need some auxiliary results. The following useful result relates to weights in the closure of the lowest alcove, $\bar{C}_{\Z}=\{\lambda\in X(T)\mid 0\leq\la \lambda+\rho,\beta^\vee\rangle\leq p\text{ for all }\beta\in R^+\}$. It is immediate from \cite[Cor.~5.4 B(i)]{BNP02}.

\begin{lemma} \label{lowalcove}
Let $G$ be simple and simply connected and suppose $L = L(\mu)$ with $\mu \in \bar{C}_{\Z}$ and $p \geq 3$. Then we have $\opH^1(G_1,L) = 0$ unless $G$ is of type $A_1$ and $L=L(p-2)$. 
\end{lemma}

In the next proposition and elsewhere, we use repeatedly the linkage principle for $G_1$, \cite[II.9.19]{Jan03}. This states that if $\Ext_{G_1}^i(L(\lambda),L(\mu))\neq 0$ for any $i\geq 0$ then $\lambda\in W_p\cdot \mu+pX(T)$, where $\cdot$ denotes the usual action of the affine Weyl group $W_p$ on $X(T)$, shifted by $\rho$.

If $V$ is a $G$-module, then as $G_1$ is a normal subgroup scheme of $G$, the cohomology group $\opH^1(G_1,V)$ inherits a $G$-module structure \cite[I.6.6]{Jan03}; since $G_1$ acts trivially on this, such an action must factor through the Frobenius morphism, hence can be untwisted to yield a $G$-module $\opH^1(G_1,V)^{[-1]}$. Of course, any simple module for $G_1$ can always be given a $G$-structure in a unique way and is associated to a unique highest weight $\lambda\in X_1(T)$, where $X_1(T)$ denotes the $p$-restricted dominant weights; see \cite[II.3.15]{Jan03}.

\begin{prop}\label{h1sless64}Let $G$ be a simple algebraic group of rank no more than $7$, let $p\geq 5$ and suppose $V$ is a restricted irreducible $G$-module of dimension at most $64$. Then either $\opH^1(G_1,V)=0$ or $(G,V,p)$ is listed in Table~\ref{tHs} (up to taking duals).\end{prop}
\begin{table}\begin{center}\begin{tabular}{l l l l l}
$G$ & $p$ & $V$ & $\dim V$ & $\opH^1(G_1,V)^{[-1]}$\\\hline
$A_1$ & $p<64$ & $L(p-2)$ &$p-1$& $L(1)$\\
$A_2$ & $p = 5,7$ & $L(p-2,1)$ & $18, 33$ & $L(1,0)$ \\
& $p=5$ & $L(3,3)$ & $63$ & $k$ \\
$B_2$ & $p = 5,7$ & $L(p-3,0)$ & $13, 54$ & $k$ \\ 
& $p = 5$ & $L(1,3)$ & $52$ & $L(0,1)$ \\
$G_2$ & $p = 7$ & $L(2,0)$ & $26$ & $k$ \\
$A_3$ & $p=5$ & $L(3,1,0)$ & $52$ & $L(1,0,0)$ \\
$A_4$ & $p=5$ & $L(1,0,0,1)$ & $23$ & $k$ \\
$A_6$ & $p=7$ & $L(1,0,0,0,0,1)$ & $47$ & $k$ \\
\end{tabular}\end{center}\caption{Non-trivial $G_1$-cohomology of irreducible modules of dimension at most $64$.}\label{tHs}\end{table}
\begin{proof}
The values of $\opH^1(G_1,V)$ where $G$ is of type $A_1$ are well-known and can be found, for example in \cite[Prop.~2.2]{SteSL2}. The statement in the cases where $G$ is of type $A_2$ or $B_2$, or type $G_2$ when $p\geq 7$ is immediate from, respectively, \cite{Yeh82} (see \cite[Prop.~2.3]{SteSL3}), \cite[Table~2]{Ye90} and \cite[Table~2]{LY93}. 

Consider the remaining cases. A list of all non-trivial restricted modules of dimension at most $64$ is available from \cite{Lub01}. We then use the $G_1$-linkage principle to remove any modules $L(\lambda)$ such that $\lambda$ is not $G_1$-linked to $\mu = 0$. Explicitly, one may calculate $w\cdot 0$ for any $w\in W$ (a finite list) and add an appopriate (uniquely defined) element of $pX(T)$ to produce a collection of $p$-restricted weights in $X_1(T)$; since we assume $\opH^1(G_1,V)\neq 0$ we know that $\lambda$ is one in this collection. Furthermore, we remove any modules $L(\lambda)$ such that $\lambda$ is in the lowest alcove, since $\opH^1(G_1,L(\lambda)) = 0$ by Lemma \ref{lowalcove}. This reduces the list of possibilities considerably. For any cases still remaining, we appeal to \cite[Thm.~3A]{BNP04-Frob} (case $r=1$), recalling $H^0(\lambda):= \Ind^G_B(\lambda)$. This implies that \begin{equation}\opH^1(G_1,H^0(\lambda))=\begin{cases}H^0(\omega_\alpha)^{[1]} &\text{ if }\lambda=p\omega_\alpha-\alpha \text{ for } \alpha \in S\\
0 &\text{ otherwise}\end{cases}.\label{bnp}\end{equation} 

Let us take $G = A_3$ and explicitly give the details. Using \cite[6.7]{Lub01}, we see that the following is a complete list of the high weights of all non-trivial restricted modules $V = L(\lambda)$ such that $\dim V \leq 64$ (up to taking duals): $(1,0,0)$, $(0,1,0)$, $(2,0,0)$, $(1,0,1)$, $(1,1,0)$, $(0,2,0)$, $(3,0,0)$, $(2,0,1)$, $(1,1,1)$, $(0,3,0)$, $(4,0,0)$, $(5,0,0)$, $(1,2,0)$ and when $p=5$ the weight $(3,1,0)$. A weight $\lambda = (a,b,c)$ is in $\bar{C}_{\Z}$ if and only if $a + b + c \leq p - 3$. So if $p \geq 11$, all weights are in $\bar{C}_{\Z}$, if $p=7$ only $(5,0,0)$ is not in $\bar{C}_{\Z}$ and if $p=5$ the first six weights in the list are in $\bar{C}_{\Z}$. 

We may reduce this list using the linkage principle for $G_1$. In our case, this implies that the only restricted weights $G_1$-linked to $(0,0,0)$ up to duals are $(p-2,p-2,p-2)$, $(p-2,1,0)$, $(p-2,p-3,0)$, $(0,p-3,2)$, $(p-4,1,p-2)$, $(p-4,0,0)$, $(p-2,2,p-2)$, $(0,p-4,0)$, $(1,p-4,1)$, $(1,p-2,1)$, $(p-3,0,p-3)$, $(p-3,0,1)$, $(p-3,p-2,1)$, $(p-2,p-2,2)$, $(p-3,2,p-3)$. By comparison with the list above, we may discount the possibility that $p=7$ and  the list of possible modules $V$ with $\opH^1(G_1,V) \neq 0$ to just $L(1,1,1)$, $L(2,0,1)$ and $L(3,1,0)$ when $p=5$.  

We now use (\ref{bnp}) to find that $\opH^1(G_1,H^0(\lambda)) = 0$ for $\lambda = (1,1,1)$, $(2,0,1)$ and $\opH^1(G_1,H^0(\lambda)) = L(1,0,0)^{[1]}$ for $\lambda = (3,1,0)$. Now, the structure of the induced modules $H^0(\lambda)$ can be deduced: each is the indecomposable reduction modulo $p$ of a certain lattice in the simple module $L_\mathbb C(\lambda)$; by comparing the weight multiplicities   in \cite{Lub01}, one finds there are just two composition factors, in each case. Since $L(\lambda)$ is the socle of $H^0(\lambda)$ one gets  $H^0(1,1,1) = L(0,1,0) | L(1,1,1)$, $H^0(2,0,1) = L(1,0,0) | L(2,0,1)$ and $H^0(3,1,0) = L(2,0,1) | L(3,1,0)$. Consider the following exact sequence for a $2$-step indecomposable $G_1$-module $M = M_1 | M_2$. 
\begin{align*}\label{seq}0&\to \opH^0(G_1,M_2)\to \opH^0(G_1,M)\to \opH^0(G_1,M_1)\nonumber\\
&\to \opH^1(G_1,M_2)\to \opH^1(G_1,M)\to \opH^1(G_1,M_1) \to \dots \end{align*}
Applying this to $H^0(\lambda)$ for $\lambda = (1,1,1)$, $(2,0,1)$ yields that $\opH^1(G_1,L(1,1,1)) = \opH^1(G_1,L(2,0,1)) = 0$. Moreover, applying the sequence to $H^0(3,1,0)$ and using the fact $\opH^1(G_1,L(2,0,1)) = 0$, it follows that $\opH^1(G_1,L(3,1,0)) \cong \opH^1(G_1,H^0(3,1,0)) \cong L(1,0,0)^{[1]}$. 
\end{proof}

Finally, we record the following result, which is presumably well-known. We were unable to locate a proof in the literature, so we give one.

\begin{lemma}\label{extA1}
Let $G$ be a simple connected algebraic group of type $A_1$, let $p>2$ and $0\leq a,b\leq p-1$. Then $\Ext^1_{G_1}(L(a),L(b)) \cong \Ext^1_{\g}(L(a),L(b))$ unless $a=b=p-1$. Moreover, $\Ext^1_{G_1}(L(a),L(b))  \neq 0$ if and only if $a = p-2-b$ and $a,b \neq p-1$, and $\Ext^1_{\g}(L(p-1),L(p-1))\cong (\g^*)^{[1]}$. 
\end{lemma} 
\begin{proof}We prove the second statement about $\Ext^1_{G_1}(L(a),L(b))$ first. Since $w\cdot a$ is either $a$ or $-a-2$, the linkage principle for $G_1$ implies that $\Ext^1_{G_1}(L(a),L(b))=0$ unless $b=a$ or $b=p-2-a$. If $b=a$ then $\Ext^1_{G_1}(L(a),L(b))=0$ by \cite[II.12.9]{Jan03} so we may now assume $b=p-2-a$. But now $\Ext^1_{G_1}(L(a),L(p-2-a))\cong\Ext^1_{G_1}(k,L(a)\otimes L(p-2-a))\cong\Ext^1_{G_1}(k,L(p-2)\oplus L(p-4)\oplus\dots\oplus L(0))$. Then the only term which survives in this expression is $\Ext^1_{G_1}(k,L(p-2))=\opH^1(G_1,L(p-2))=\opH^1(G_1,H^0(p-2))$ whose structure as a $G$-module can be read off from \cite[II.12.15]{Jan03} or \cite[Thms.~3(A-C)]{BNP04-Frob}.

Now, in sequence (\ref{seq}), put $V=L(b)\otimes L(a)^*$. Then we have an isomorphism $\Ext^1_{G_1}(L(a),L(b)) \cong \Ext^1_{\g}(L(a),L(b))$ if we can show that $\Hom^s(\g,V^\g)=0$. But if $V^\g$ were non-zero then we must have $L(a)\cong L(b)$ and $V^\g\cong k$. Thus we now assume $a=b$.

The assignation of $V$ to the sequence (\ref{seq}) is functorial, thus, associated to the $G$-map $k\to V$, there is a commutative diagram
\[\begin{CD}0@>>> \opH^1(\g,k)=0 @>>> \Hom^s(\g,k^\g)\cong(\g^*)^{[1]} @>\cong>> \opH^2(G_1,k)\\
@. @VVV @V\cong VV @V\theta VV \\
0 @>>> \opH^1(\g,V) @>>> \Hom^s(\g,V^\g)\cong(\g^*)^{[1]} @>\zeta >> \opH^2(G_1,V)\end{CD},\] where the natural isomorphism $k^\g\to V^\g$ induces the middle isomorphism. We want to show that $\zeta$ is injective, since then it would follow that $\opH^1(\g,V)=0$. To do this it suffices to show that $\theta$ is an injection $(\g^*)^{[1]}\to \opH^2(G_1,V)$ and for this, it suffices to show that the simple $G$-module $(\g^*)^{[1]}$ does not appear as a submodule of $\opH^1(G_1,V/k)$. If $(\g^*)^{[1]}$ did appear as a submodule, there must be a composition factor $L(\nu)$ of $V/k$ such that $\opH^1(G_1,L(\nu))\cong (\g^*)^{[1]}$. Writing $L(\nu)=L(\nu_0)\otimes L(\nu_1)^{[1]}$ using Steinberg's tensor product theorem, we have $\opH^1(G_1,L(\nu))\cong \opH^1(G_1,L(\nu_0))\otimes L(\nu_1)^{[1]}$ and the latter is non-zero only when $\nu_0=p-2$. Since the weights of $V$ are all even and bounded above by $2p-2$, with equality if and only if $a=b=p-1$, we must have $\nu_0=p-2$, $\nu_1=1$ and indeed, $a=b=p-1$. But then $\opH^2(G_1,V)\cong \Ext_{G_1}^2(L(p-1),L(p-1))$ which is zero since $L(p-1)$ is the projective Steinberg module. Thus from the bottom line of the diagram we have an isomorphism $\opH^1(\g,V)\cong (\g^*)^{[1]}$ as required. 
\end{proof}

\subsection{Nilpotent orbits}

At various points we use the theory of nilpotent orbits, particularly the results of \cite{Pre95}. Everything we need can be found in \cite[\S1-5]{Jan04}. We particularly use the fact that a nilpotent element $e$ has an associated cocharacter; that is a homomorphism $\tau:\Gm\to G$ such that under the adjoint action, we have $\tau(t)\cdot e=t^2e$ and $\tau$ evaluates in the derived subgroup of the Levi subgroup in which $e$ is distinguished. Recall that an $\sl_2$-triple is a triple of elements $(e,h,f)\in \g\times\g\times\g$ such that $[h,e]=2e$, $[h,f]=-2f$ and $[e,f]=h$. In the case that a nilpotent element $e$ is included in an $\sl_2$-triple in $\g$, the theory of associated cocharacters can be used to prove the following useful result.

\begin{prop}[cf.~{\cite[Prop.~3.3(iii)]{HSMax}}]\label{findingcoch}
Suppose the nilpotent element $e$ is not in an orbit containing a factor of type $A_{p-1}$, and that $h$ is a toral element in the image of $\ad e$ with $[h,e]=2e$. Then there is a cocharacter $\tau$ associated to $e$ with $\Lie(\tau(\Gm))=\la h\ra$.
\end{prop}

We also need to use the Jordan block structure of nilpotent elements on the adjoint and minimal modules in good characteristic. For the adjoint modules we may use \cite{Law95}, and see \cite{PreSteNilp} which provides the validity of these tables for the nilpotent analogues of the unipotent elements considered there. For the minimal modules we may use \cite[Thm~1.1]{SteComp}. 

When referring to nilpotent orbits of $\Lie(G)$ for a simple algebraic group $G$, we use the labels defined in \cite{LS12}. In particular, when $G$ is of exceptional type these labels are described in Chapter 9 of [loc. cit].  

At certain points we make use of the notion of a reachable nilpotent element. A (nilpotent) element $e\in\g$ is said to be \emph{reachable} if it is contained in the derived subalgebra of its centraliser. That is $e\in[\g_e,\g_e]$. The reachable elements of $\g$ have all been classified in \cite{PreSteNilp}. Then the main point is that long root elements in simple subalgebras are almost always reachable in those subalgebras, hence are reachable elements of $\g$. We will need the following result. 

\begin{lemma}\label{reachlem}Let $\h=\Lie(H)$ for $H$ a simple algebraic group, not of type $A_1$. Then any long root element $e$ is reachable, except possibly if $H$ is of type $C_n$ and $p=2$.\end{lemma}
\begin{proof} Since all long root elements are in a single $H$-orbit, it suffices to prove the lemma in the case $e=e_{\tilde\alpha}$ for $\tilde\alpha$ the highest root. Then it suffices to find two root elements $\alpha$ and $\beta$ with $\alpha+\beta=\tilde\alpha$ such that $[e_\alpha,e_\beta]\neq 0$.

This is a simple case-by-case check of the root systems. In all cases except $C_n$, one can take $\alpha$ and $\beta$ to be long roots. For a Chevalley basis, we have $[e_\alpha,e_\beta]=\pm e_{\tilde\alpha}$ and so we are done. In case $C_n$, one may take two short roots $\alpha$ and $\beta$, and one has $[e_\alpha,e_\beta]=\pm 2\cdot e_{\tilde\alpha}$, which is non-zero provided $p\neq 2$.
\end{proof}

\section{Irreducible and completely reducible subalgebras of classical Lie algebras. The proof of Theorem \ref{ssgcrcorollary}.}\label{sec:gcrclassical}

For the time being, assume $p>2$. In this section we show that Theorem \ref{ssgcrcorollary} holds in the case $G$ is simple and classical; this is Proposition \ref{thm:gcrclassical} below. Let $G$ be a simple, simply-connected algebraic group of classical type and let $\h$ be a subalgebra of $\g=\Lie(G)$. We first give a condition for $\h$ to be $G$-irreducible. That is, that $\h$ is in no proper parabolic subalgebra of $\g$. This is given in terms of the action of $\h$ on the natural module $V$ for $G$, as it is in the group case---see \cite[Lem.~2.2]{LT04}. 

\begin{prop}\label{girclassical}The algebra $\h$ is $G$-irreducible if and only if one of the following holds:
\begin{enumerate}\item $\g=\sl(V)$ and $\h$ acts irreducibly on $V$;
\item $\g=\sp(V)$ or $\so(V)$ and $\h$ stabilises a decomposition $V\cong V_1\perp V_2\perp \dots \perp V_n$, where the $V_i$ are a set of mutually orthogonal, non-degenerate and non-isomorphic $\h$-irreducible submodules of $V$.
\end{enumerate}\end{prop}
\begin{proof} By \cite[Prop~12.13]{MT}, the (proper) parabolic subgroups of $G$ are precisely the stabilisers of (non-trivial) flags $\FF^\bullet$ of totally isotropic subspaces of $V$ (where $G = \SL(V)$ preserves the $0$-form on $V$). Let $\FF^\bullet$ be a flag of subspaces such that the $k$-points of its stabiliser is a parabolic subgroup $P$. We claim that $\Stab(\FF^\bullet)$ is smooth. We certainly have $\Lie(\Stab(\FF^\bullet))$ contains $\Lie(P)$ so that $\Lie(\Stab(\FF^\bullet))$ is maximal rank, thus corresponds to a subsystem of the root system. Any root space $\langle e_{\alpha}\rangle$ contained in $\Lie(\Stab(\FF^\bullet))$ gives rise to a root subgroup of $\Stab(\FF^\bullet)(k) = P$, via $t \mapsto \exp(t.e_{\alpha})$. Thus $\Lie(\Stab(\FF^\bullet)(k)) = \Lie(\Stab(\FF^\bullet))$, as required.   

The case $G=\SL(V)$ is now clear since $\h$ fixes a non-trivial subspace of $V$ if and only if it is contained in a parabolic subalgebra. 

Now suppose $G=\Sp(V)$ or $\SO(V)$. Firstly, let $\h$ be a $G$-irreducible subalgebra of $\g$ and suppose $V_1$ is a minimal non-zero $\h$-invariant subspace of $V$, so $V_1$ is an $\h$-irreducible submodule of $V$. Then $V_1$ must be non-degenerate or else $\h$ would stabilise a non-zero totally isotropic subspace and hence be contained in a proper parabolic subalgebra. We then use an inductive argument applied to $V_1^\perp$ to see that $\h$ stabilises a decomposition $V\cong V_1\perp V_2\perp \dots \perp V_n$ of non-degenerate, mutually orthogonal, $\h$-irreducible submodules. If $V_i|_\h\cong V_j|_\h$ by an isometry $\phi:V_i\to V_j$ for $i\neq j$ then $\h$ preserves the totally isotropic subspace $\{v_i+i\phi(v_j)\}\subset V_i\oplus V_j$ (where $i^2=-1$). Thus the $V_i$ are pairwise non-isomorphic. Finally, it remains to note that that any subalgebra $\h$ preserving such a decomposition as in (ii) is $G$-irreducible since it stabilises no totally-isotropic subspaces of $V$ by definition. 
\end{proof}

Since the Levi subalgebras of classical groups are themselves classical one gets the following, using precisely the same argument as in \cite[Ex.~3.2.2(ii)]{Ser05}. (We remind the reader of our assumption that $p >2$.)

\begin{lemma}\label{gcrOnNatural} The subalgebra $\h$ of $\g$ is $G$-cr if and only if it acts completely reducibly on the natural module $V$ for $\g$.\end{lemma}

To prove the next proposition, we use Lemma \ref{gcrOnNatural} together with the following non-trivial result.
\begin{theorem}[{\cite[Cor.~8.12]{HS14}}]\label{liealgcompred}Let $G$ be a semisimple algebraic group and let $V$ be a $\g$-module with $p>\dim V$.
Assume that $\g = [\g,\g]$. Then $V$ is semisimple.\end{theorem}

\begin{prop}\label{thm:gcrclassical}Let $G$ be a simple algebraic group of classical type with $h(G)$ its Coxeter number and $\g$ its Lie algebra. If $p>h(G)$ then any semisimple subalgebra $\h$ of $\g$ is $G$-cr.\end{prop} 
\begin{proof}
Let $G$ be a simple algebraic group of classical type with $p>h(G)$. Now assume, looking for a contradiction, that $\h$ is a non-$G$-cr subalgebra of $\g$. Thus $\h$ is in a non-trivial parabolic subalgebra, projecting isomorphically to some proper Levi subalgebra $\l$ with Coxeter number $h_1<h(G)$. One checks that the condition $p>h(G)$ implies that $p>\dim V$ for $V$ a minimal-dimensional module for any simple factor of $\l$. Now \cite[Main Thm.]{Str73} implies that the projection of $\h$ to the simple factors of $\l$ are all direct products of classical-type Lie algebras, hence $\h$ itself is a direct product of classical-type Lie algebras, isomorphic to $\Lie(H)$ for $H$ some semisimple algebraic group. Thus by Theorem \ref{liealgcompred} we have that $\h$ acts completely reducibly on $V$. Hence by Lemma \ref{gcrOnNatural} we have that $\h$ is $G$-cr.\end{proof}

\section{$G$-complete reducibility in exceptional Lie algebras. Proof of Theorem \ref{maintheorem}.}\label{sec:proofofmain}
Let $G$ be reductive and $P$ be a parabolic subalgebra of $G$ with Levi decomposition $P=LQ$. We begin with some general results on $G$-complete reducibility of subalgebras. For our purposes, they will be used in order to generate examples of non-$G$-completely reducible subalgebras.

\begin{lemma}[\!{\cite[Lem.~4]{McN07}}] \label{levinongcr} Let $G$ be a reductive algebraic group and let $L$ be a Levi subgroup of $G$. Suppose $\h \subseteq \Lie(L)$ is a Lie subalgebra. Then $\h$ is $G$-cr if and only if $\h$ is $L$-cr. \end{lemma}

\begin{lemma}[\!{\cite[Thm.~5.26(i)]{BMRT13}}] \label{hnongcr} Let $G$ be a reductive algebraic group and let $P = LQ$ be a parabolic subgroup of $G$. Suppose $\h$ is a subalgebra of $\g$ contained in $\p = \Lie(P)$. Then if $\h$ is not $Q$-conjugate to a subalgebra of $\l = \Lie(L)$ then $\h$ is non-$G$-cr.\end{lemma}

The following lemma provides a strong connection between the structure of modular Lie algebras and the notion of $G$-complete reducibility and will be used very often.

\begin{lemma}\label{psubalgebraornongcr}Let $G$ be a reductive algebraic group and suppose $p$ is a good prime for $G$.  Suppose further that $\h$ is a simple $G$-cr subalgebra of $\g$ which is restrictable as a Lie algebra. Then either $\h$ is a $p$-subalgebra of $\g$ or $\h$ is $L$-irreducible in a Levi subalgebra $\l = \Lie(L)$ of $\g$ with a factor of type $A_{rp-1}$ for some $r\in\mathbb{N}$.\end{lemma}
\begin{proof}If $\h$ is not a $p$-subalgebra of $\g$ then its $p$-closure, $\h_p$, is a $p$-envelope of $\h$ strictly containing $\h$. Since $\h$ is restrictable, by \cite[Thm.~2.5.8]{SF88} $\h_p$ has structure $\h_p=\h\oplus J$ for $J$ an ideal of $\h_p$ centralised by $\h$. Now suppose $L$ is chosen minimally such that $\h\subseteq \l$; as $\h$ is $G$-cr, it follows that $\h$ is $L$-irreducible. If $p$ is a very good prime for $\l$, then $\l=\l'\oplus \z(\l)$ where $\l'$ is the direct sum of simple Lie algebras and a central torus $\z(\l)$, both of which are $p$-subalgebras. Since $\h$ is simple, it has trivial projection to $\z(\l)$ and so $J\subseteq \l'$. But since $p$ is good, the centraliser of an element of the semisimple Lie algebra $\l'$ is in a proper parabolic of $\l'$ and $\h$ is $G$-cr; thus $\h$ is in a proper Levi subalgebra of $\l'$ and so $L$ was not minimal, a contradiction. It follows that $p$ is not a very good prime for $\l'$, but this precisely means that $L$ contains a factor of type $A_{rp-1}$.\end{proof} 

\begin{cor}\label{kwcor} Let $\g=\Lie(G)$ be an exceptional Lie algebra in good characteristic. Suppose $\h=\Lie(H)$ for $H$ a simple algebraic group not of type $A_1$ and that $\h$ is a non-$G$-cr subalgebra of $\g$. Choose $\p=\l+\q$ minimal subject to containing $\h$. Then the projection $\bar\h$ of $\h$ to $\l$  is a $p$-subalgebra.\end{cor}
\begin{proof}By the proof of Lemma \ref{psubalgebraornongcr}, if $\bar\h$ is not a $p$-subalgebra then $\l$ is of type $A_{p-1}$ and $\bar\h$ projects to a subalgebra of $\sl_p$ acting irreducibly on the natural module. Since $\h$ is not of type $A_1$, the Kac-Weisfeiler conjecture (which is true for this situation by \cite{Pre95KW}) implies the only non-restricted representations of $\h$ have dimension at least $p^2$. This is a contradiction.\end{proof}

Unless otherwise mentioned, for the rest of this section $G$ will denote an exceptional algebraic group over an algebraically closed field $k$ of characteristic $p\geq 5$ and we will let $\g = \Lie(G)$. Our strategy to prove Theorem \ref{maintheorem} is as follows---we use the opportunity to fix notation for the remainder of the paper as we explain this: 

Suppose there exists a non-$G$-cr simple subalgebra $\h$ of $\g$ of a given type. Then $\h$ must be contained in a parabolic subalgebra $\p=\l+\q=\Lie(P)=\Lie(LQ)$, which \emph{from now on, we assume is minimal subject to containing $\h$}. It follows that the projection $\bar\h$ of $\h$ to $\l$ is $L$-irreducible. Since $\h$ is not conjugate to a subalgebra of $\l$, Proposition \ref{h1ofqis0} implies that $\q$ contains an $\h$-composition factor $V$ for which $\opH^1(\h,V)\neq 0$. If $\h$ is not isomorphic to $\psl_p$ or $W_1$ then $\h\cong\Lie(H)$ for some simple simply-connected algebraic group $H$ of the same type and the remarks at the beginning of \S\ref{sec:h1s} imply that $\opH^1(H_1,V)\neq 0$ for $H_1$ the first Frobenius kernel of $H$. In fact the same isomorphism holds with $H=\SL_p$ and $\h=\psl_p$ as any $\psl_p$-module can be lifted to a module for $\sl_p$ by allowing the centre to act trivially; then one may apply the exact sequence of \cite[7.5.3]{Wei94}. By Lemma \ref{dimofparabs} we must have that the dimension of $V$ is less than $64$. By Proposition \ref{h1sless64} it now follows that: \begin{equation}\label{VinTable}\text{Either $\h\cong W_1$ or $V$ appears in Table~\ref{tHs}.}\end{equation} By analysing the structure of $\q$ closely, we will find that in most cases no such $V$ appears as a composition factor of $\q$ so that these cases are ruled out: see Lemmas \ref{nobadthingsinclassicalLevis} and \ref{rankatleast3} below. 

One set of cases requires more work: If $\l$ is a Levi subalgebra of $\g$ of type $E_6$ or $E_7$, then we  investigate all possible actions of $\h$ on the smallest dimensional modules for $E_6$ and $E_7$ (of dimensions $27$ and $56$, respectively). We will see that in any such action, a regular nilpotent element of $\h$ does not act consistently with the Jordan block sizes of nilpotent elements on the relevant modules, as described in \cite{SteComp}. Having reduced the possible cases as above, we show that the remaining cases of $(G,\h,p)$ do indeed give rise to non-$G$-cr subalgebras, recorded in Lemmas \ref{mtfora1a}, \ref{mtforb2}, \ref{mtforg2} and \ref{mtforw1}.

\begin{lemma} \label{nobadthingsinclassicalLevis}
Suppose $\h$ is a simple non-$G$-cr subalgebra not of type $A_1$ or $W_1$ and that $\p = \l + \q$ is a minimal parabolic containing it.  Then $(G,\h,p)$ is $(E_8,B_2,5)$, $(E_7, G_2,7)$ or ($E_8, G_2,7)$; or $\l'$ has a simple factor of type $E_6$ or $E_7$.  
\end{lemma}

\begin{proof}
Without loss of generality we may assume $G$ is simply-connected. Suppose $P=LQ$ with $\Lie(P)=\p$ and let $L'=L_1\dots L_r$ be an expression for the derived subgroup of $L$ as a central product of simple factors not containing any exceptional factors. As each $L_i$ is simply connected, we may write $\l' = \l_1\oplus \dots\oplus \l_r$ with $\l_i=\Lie(L_i)$. Since $P$ was minimal, we must have that the projection of $\h$ to each $\l_i$ is $L_i$-irreducible; call this $\h_i$. 

Since $\h$ is not of type $A_1$, all of the $L_i$ factors have rank at least $2$. As $L_i$ is classical by assumption, it has a natural module $V_i$ and it follows from Lemma \ref{girclassical} and Corollary \ref{kwcor} that there exists an $L_i$-irreducible restricted subgroup $H_i$ whose action on $V_i$ differentiates to that of $\h_i$. The action of $\h_i$ on $V_i$ determines it up to $L_i$-conjugacy, except if $L_i$ is of type $D_n$ and there are two conjugacy classes interchanged by a graph automorphism (or up to three if $L_i$ is of type $D_4$). Thus we may write $\h=\Lie(H)$ for $H$ a diagonal subgroup of $H_1 \dots H_r$. One may now compute a list of possible $\h$-factors of $\q$ by differentiating the $H$-factors on $Q$; the latter are available from \cite[Lem.~3.4]{LS96}. By (\ref{VinTable}) this implies $(H,p,\lambda) = (B_2,5,(2,0))$, $(B_2,5,(1,3))$ or $(G_2,7,(2,0))$. 

Suppose $H$ is a subgroup of type $B_2$. The only Levi factors for which $L(2,0)$ or $L(1,3)$ occur as an $H$-composition factor are $A_3 A_4$ and $D_7$. Hence $G$ is of type $E_8$. Similarly, if $H$ is a subgroup $G_2$ then the only Levi factors for which $L(2,0)$ occurs as an $H$-composition factor are $A_6$ and $D_7$, and so $G$ is of type $E_7$ or $E_8$. 
\end{proof}

The next lemma reduces the proof of Theorem \ref{maintheorem} to considering all simple subalgebras of rank at most $2$.

\begin{lemma}\label{rankatleast3}Suppose $\h$ is a simple subalgebra of rank at least $3$. Then $\h$ is $G$-cr.\end{lemma}
\begin{proof}
Statement (\ref{VinTable}) above implies that we are done unless $\h$ is of type $A_3$, $A_4$ or $A_6$. Firstly, suppose $\h$ is of type $A_3$ and is a non-$G$-cr subalgebra of $\g$. Then (\ref{VinTable}) tells us $p=5$ and either $V:=L(3,1,0)$ or its dual is an $\bar\h$-composition factor of $\q$. Now $\dim V=52$ and so we are forced to conclude that $G$ is of type $E_8$. By Lemma \ref{nobadthingsinclassicalLevis} and dimensions again, we must have $\l'$ of type $E_7$. But the only non-trivial factor of $\q$ has dimension $56$. If $V$ were a composition factor of the self-dual module $V_{56}$, then so would be its dual; a contradiction by dimensions.

Now suppose $\h$ is of type $A_4$. Statement (\ref{VinTable}) tells us $p=5$ and $\q$ contains an $\bar\h$-composition factor $V:=L(1,0,0,1)$. By Lemma \ref{nobadthingsinclassicalLevis} we may also assume that $G$ is of type $E_7$ or $E_8$ with $\l'$ chosen minimally of type $E_6$ or $E_7$. If $\l'$ is of type $E_6$, the non-trivial $\l'$-composition factors of $\q$ are either $V_{27}$ or its dual. Since $V$ appears amongst $\q\downarrow\bar\h$, the $\bar\h$-composition factors of $V_{27}$ are $L(1,0,0,1) / k^4$. Since the restriction of a natural module $L(1,0,0,0)$ for $\h$ to a Levi $\sl_2$-subalgebra has composition factors $L(1)/k^3$, and $L(1,0,0,1)$ is a composition factor of $L(1,0,0,0)\otimes L(0,0,0,1)$, we may calculate that the restriction to a Levi $\sl_2$-subalgebra of $V_{27}$ is a completely reducible module with composition factors $L(2)/L(1)^6/k^{12}$. A non-zero nilpotent element of this subalgebra acts with Jordan blocks $3+2^6+1^{12}$, though this is impossible by \cite[Table~4]{SteComp}. In case $\l'$ is of type $E_7$, we see that $V_{56}$ contains a composition factor $L(1,0,0,1)$ and the remaining composition factors must have dimension $33$ or less, and if not self-dual, must have dimension $16$ or less. Up to duals, the possible composition factors together with their restrictions to a Levi subalgebra $\s$ of type $\sl_2$ are in the following table.
\begin{center}\begin{tabular}{l l}
$\lambda$ & $L(\lambda)\downarrow\s$\\\hline
$(1,0,0,1)$ & $L(2)+L(1)^6+k^8$\\
$(1,0,0,0)$ & $L(1)+k^3$\\
$(0,1,0,0)$ & $L(1)^3+k^4$\end{tabular}\end{center}
The restriction of any resulting module to $\s$ is completely reducible, and so the Jordan blocks of a non-zero nilpotent element $e\in\s$ are determined by the $\bar\h$-composition factors on $V_{56}$. It is easily checked that there is no way of combining these composition factors compatibly with the possibilities in \cite[Tables~2,~3]{SteComp}.

Finally, suppose that $\h$ is of type $A_6$. Arguing in a similar fashion, we find that $G$ is of type $E_8$ with $p=7$; the type of $\l'$ is $E_7$; and the $\bar\h$-composition factors on $V_{56}$ are $L(\varpi_1 + \varpi_6)/k^9$. Restricting to a Levi $\sl_2$-subalgebra and comparing once again with \cite[Tables~2,~3]{SteComp} yields a contradiction.
\end{proof}

\addtocontents{toc}{\protect\setcounter{tocdepth}{1}}
\subsection{Subalgebras of type $A_1$}\label{sec:a1s}
In this section we show that Theorem \ref{maintheorem} holds in case $\h$ is of type $A_1$. The following also deals with one direction of Theorem \ref{ssgcrcorollary}.

\begin{lemma}\label{mtfora1a}Let $G$ be a reductive algebraic group and let $p>2$. Then whenever $p\leq h(G)$, there exists a non-$G$-cr $\sl_2$-subalgebra containing a regular nilpotent element $e$, a toral element $h$ and an element $f$ regular in a proper Levi subalgebra of $\g$.\label{a1nongcr}\end{lemma}
\begin{proof} It suffices to tackle the case where $G$ is simple. Let $\g_\Z$ be a lattice defined via a Chevalley basis in the simple complex Lie algebra $\g_\C$ of the same type as $\g$ and let $e$ be the regular element given by taking a sum of all simple root vectors in $\g_\Z$. Then there is an element $f\in\g_\Z$ which is a sum of negative root vectors such that $(e,h,f)$ is an $\sl_2$-triple. These are easily constructed in the case $G$ is classical and are given explicitly by  \cite[Lem.~4]{Tes92} in the case $G$ is exceptional:
\begin{enumerate}
\item $\g = G_2$, $f = 6 e_{-\alpha_1}  + 10 e_{-\alpha_2}$;
\item $\g = F_4$, $f = 22 e_{-\alpha_1} + 42 e_{-\alpha_2} + 30 e_{-\alpha_3} + 16 e_{-\alpha_4}$;
\item $\g = E_7$, $f = 34 e_{-\alpha_1} + 49 e_{-\alpha_2} + 66 e_{-\alpha_3} + 96 e_{-\alpha_4} + 75 e_{-\alpha_5} + 52 e_{-\alpha_6} + 27 e_{-\alpha_7}$;
\item $\g = E_8$, $f = 92 e_{-\alpha_1} + 136 e_{-\alpha_2} + 182 e_{-\alpha_3} + 270 e_{-\alpha_4} + 220 e_{-\alpha_5} + 168 e_{-\alpha_6} + 114 e_{-\alpha_7} + 58 e_{-\alpha_8}$.
\end{enumerate}

Reducing everything modulo $p>2$ gives an $\sl_2$-subalgebra $\h$ of $\g$.
Moreover, since $p < h(G)$ we have that $e^{[p]} \neq 0$, which follows from the description of the Jordan blocks of regular nilpotent elements in \cite[\S2]{Jan04} for $G$ classical, and is immediate from \cite[Table~D]{Law95} in case $G$ is exceptional. Therefore, in each case $\h$ is a non-$p$-subalgebra of type $A_1$ (noting that there is only one $p$-structure on $\h$ by \cite[Cor.~2.2.2(1)]{SF88}). Since $E_6$ contains $F_4$ as a $p$-subalgebra, the non-$p$-subalgebra $\h$ contained in $F_4$ is also a non-$p$-subalgebra of $E_6$. 

Suppose $G$ is not of type $A_{p-1}$. Then since $\h$ contains a regular nilpotent element of $\g$, it is certainly not contained in a Levi subalgebra of type $A_{rp-1}$, hence being non-$p$-subalgebras, these subalgebras are non-$G$-cr, by Lemma \ref{psubalgebraornongcr}. In particular each is in a proper parabolic subalgebra of $\g$. Thus $f$ is no longer regular in $\g$. This can be seen explicitly in the case $\g$ is exceptional as $p<h(G)$ implies that $p$ divides at least one of the coefficients of the root vectors of $f$. In the classical case, $e$ acts as a single Jordan block on the natural module in types $A$, $B$ and $C$ acting on standard (orthonormal) basis vectors as $e(e_i)=e_{i-1}$ whereas $f(e_i)=\lambda_i e_{i+1}$ for $\lambda_i$ the coefficient of the simple root vector $e_{-\alpha_i}$. Since $\h$ is in a proper parabolic subalgebra, it stabilises a subspace, meaning that some non-zero collection of the $\lambda_i$ must be congruent to $0$ modulo $p$. The remainder determine a regular element in some proper Levi subalgebra obtained by removing the appropriate nodes of the Dynkin diagram. In type $D_n$, the regular nilpotent element $e$ acts with Jordan blocks of size $2n-1$ and $1$ on the natural module, and only regular elements act in such a way. Since $V|_{\so_{2n-1}}\cong V'\oplus k$ for $V'$ the natural module for $\so_{2n-1}$, and a regular element of the latter acts as a full Jordan block on $V'$, we must have that $e$ is contained in subalgebra of type $\so_{2n-1}$. Hence we may embed $e$ in a regular $\sl_2$-subalgebra $\h$ in the $\so_{2n-1}$ subalgebra such that $h$ is toral and $f$ is regular in a Levi subalgebra of $\so_{2n-1}$. This implies that $f$ is regular in a Levi subalgebra of $\g=\so_{2n}$. Finally to check the theorem in case $G$ of type $A_{p-1}$, simply observe from Proposition \ref{h1sless64} that $\sl_2$ has indecomposable representations of dimension $p$ of the form $k|L(p-2)$, which can be used to embed $\h$ in $\g$ such that $\h$ does not act completely reducibly on the natural module for $G$ with $e$ being regular but $f$ not. The subalgebra $\h$ is then non-$G$-cr by Lemma \ref{gcrOnNatural}.

(Specifically, we may take 

\[e:=\left(\begin{array}{cccccccccc}
0&1&0&0&\dots&0&\\
&0&1&0&\ddots&0\\
&&0&1 & \ddots&0\\
&&&0&\ddots&0\\
&&&&\ddots&1&\\
&&&&&0\\
\end{array}\right),f:=\left(\begin{array}{cccccccccc}
0&0&0&0&\dots&1\\
\lambda_1&0&0&0&\ddots&0\\
&\lambda_2&0&0 & \ddots&0\\
&&\ddots&0&\ddots&0\\
&&&\lambda_{p-2}&\ddots&0&\\
&&&&0&0\\
\end{array}\right) \]
where $\lambda_i = -i(i+1) \mod p$. Thus $f$ is a regular nilpotent element of type $A_{p-2}$.)
\end{proof}

\begin{lemma}\label{mtfora1b}Theorem \ref{maintheorem} holds in the case $\h$ is of type $A_1$.\end{lemma}
\begin{proof}The case $p\leq h(G)$ is supplied by Lemma \ref{a1nongcr}, so suppose $p > h(G)$ and $\h$ is a non-$G$-cr subalgebra of type $A_1$. By statement (\ref{VinTable}), there is an $\bar\h$-composition factor of $\q$ isomorphic to $L(p-2)$, of dimension $p-1$. In each case we will show this is impossible. 

When $\g$ is of type $G_2$, we have $h(G) = 6$ and the largest dimension $\l'$-composition factor occurring is {$4$-dimensional}. Hence $\g$ is not of type $G_2$.

When $\g$ is of type $F_4$, we have $h(G) = 12$. Only a Levi subalgebra of type $C_3$ has an $\l'$-composition factor of dimension $12$ or more; the composition factor is $L(0,0,1)$, which is $14$-dimensional. Using Proposition \ref{girclassical} we see that a Lie algebra of type $C_3$ contains two subalgebras of type $A_1$ not contained in parabolics when $p > 12$, acting either as $L(5)$ or $L(3) + L(1)$ on the natural module. Since $\bigwedge^3(L(1,0,0)) \cong L(0,0,1) + L(1,0,0)$ one calculates that the composition factors of such $\sl_2$ subalgebras on $L(0,0,1)$ are $L(9) / L(3)$ and $L(5) / L(3)^2$, respectively. In particular, neither has a composition factor $L(p-2)$ for $p > 12$ and so $\g$ is not of type $F_4$.

When $\g$ is of type $E_6$, we also have $h(G)=12$. Since all of the Levi subalgebras of $\g$ are of classical type, we use Proposition \ref{girclassical} to find the $L$-irreducible subalgebras of type $A_1$. As in the $F_4$ case, it is straightforward to check that none of them has a composition factor $L(p-2)$ when $p > 12$. (To find restrictions of the spin modules for Levi subalgebras of type $D$, one uses \cite[Prop.~2.13]{LS96}.)

When $\g$ is of type $E_7$, we have $h(G) = 18$. The same approach as used above rules out $\bar\h \subseteq \l$ for $\l$ consisting of classical Levi factors. Suppose $\l'$ is of type $E_6$ and that $\bar\h$ is a subalgebra of type $A_1$ with a composition factor $L(p-2)$ on $V_{27}$. Then the action of a regular nilpotent element of $\bar\h$ on $V_{27}$ has a Jordan block of size at least $p-1 \geq h(G) = 18$. This is a contradiction, since \cite[Table~4]{SteComp} shows that the Jordan blocks of the action of any nilpotent element of $E_6$ have size at most 17. 

For $\g$ of type $E_8$, we have $h(G) = 30$. As above, one similarly rules out the cases where $L$ is classical or type $E_6$. So suppose $L$ is of type $A_1 E_6$. Restricting to the first factor, the only composition factors are trivial or isomorphic to $L(1)$, while the composition factors for the second factor on $\q$ are trivial, or isomorphic to $V_{27}$ or its dual. Since $\bar\h$ has a composition factor of high weight $p-2$ on $\q$ it follows that $\bar\h$ acts on $V_{27}$ with a composition factor of high weight $p-3$ or $p-2$. Therefore, the action of a regular nilpotent element $e$ of $\bar\h$ on $V_{27}$ has a Jordan block of size at least $p-2 \geq h(G)-1 = 29$. This contradicts \cite[Table~4]{SteComp}, as before. Finally, if $\l'$ is of type $E_7$, then $V_{56}$ has an $\bar{\h}$-composition factor $L(p-2)$. The action of a regular nilpotent element $e$ of $\bar \h$ has a Jordan block of size at least $p-1 \geq h(G) = 30$ on $V_{56}$. Using \cite[Tables~2,~3]{SteComp}, we see this is a contradiction since the largest Jordan block of a nilpotent element acting on $V_{56}$ is $28$. \end{proof}
 
\subsection{Subalgebras of rank $2$}
In this section we prove Theorem \ref{maintheorem} holds for all simple subalgebras of $\g$ of rank $2$ in a series of lemmas, taking each isomorphism class of $\h$ in turn. First of all we turn to a very special case.

\begin{lemma}\label{a2max}Let $p=5$, $\g=E_7$ and $\h\cong\sl_3$ be a $p$-subalgebra of $\g$. Suppose the highest root element $e=e_{\tilde\alpha} \in \h$ is nilpotent of type $A_4+A_1$. Then $\h=\Lie(H)$ for $H$ a maximal closed connected subgroup of type $A_2$ in $G$.\end{lemma}

\begin{proof}One has in $\h$ the relation $[e_{\alpha},e_{\beta}]=e$ holds, with $e_{\alpha},e_{\beta}\in\h_{e}\subset\g_{e}\subseteq\g(\geq 0)$. Now one sees from \cite{LT11} that $\g_e(0)$ is toral, hence the nilpotence of $e_{\alpha}$ and $e_\beta$ imply that they are in fact contained in $\g_e(>0)$. The projections $\overline{e_{\alpha}},\overline{e_{\beta}}$ to $\g_e(1)$ are hence non-trivial and we must have $[\overline{e_{\alpha}},\overline{e_{\beta}}]=e$. Recall $e$ is contained in an $\sl_2$-triple $(e,h_{\tilde\alpha},f_{\tilde\alpha})\in\h\times\h\times\h$. Now $[h_{\tilde\alpha},e]=2e$ and $h_{\tilde\alpha}$ is in the image of $\ad e$. Since $e$ has a factor of type $A_{p-1}$ in $\g$ there is more than one element $h_{\tilde\alpha}$ satisfying these conditions. Indeed if $h_1\in\Lie(\tau(\Gm))$ has this property, then so does $h':=h_1+\lambda h_0$ for $h_0\in\z(\sl_5)$, where the projection of $e$ to its $A_4$ factor is regular in $\sl_5$ (for more details, see the argument in \cite[\S A.2]{HSMax}). The cases where $\lambda\in\F_5$ give all such instances where $h'$ is also a toral element, therefore we may restrict to these five cases. Furthermore, the cases where $\lambda$ is $1$ or $2$ respectively, are conjugate to the cases where $\lambda$ is $4$ or $3$ respectively, via a representative of an element of the Weyl group of $E_7$ inducing a graph automorphism on the $A_4$ factor, and centralising the $A_1$ factor (\!\!\cite[Table~10]{Car72}). Now it is an easy direct calculation using the elements given in \cite[Table 4]{HSMax} that if $\lambda\neq 0$ then a basis for the elements of $\g_e(1)$ on which $h'$ has weight $1$ is $\def\arraystretch{0.5} \arraycolsep=0pt\left\{e_{\Small \begin{array}{c c c c c c c}-0&0&0&0&0&1\\&&0\end{array}},e_{\Small\begin{array}{c c c c c c c}0&0&0&0&1&1\\&&0\end{array}}\right\}$. One checks that the commutator of these two is $e_{\alpha_6}$. The latter is not of type $A_4+A_1$. Hence we conclude $\lambda=0$.

As $h_{\tilde\alpha}\in\tau(\Gm)$, there is a standard $\sl_2$-triple $(e,h_{\tilde\alpha},f)$, that is a regular $\sl_2$-subalgebra of a Levi subalgebra of type $A_4A_1$. We have $f_{\tilde\alpha}-f\in\g_e\subseteq\g_e(\geq 0)$ so that $f_{\tilde\alpha}$ projects to $f$ in $\g(-2)$. Now $\tau^{-1}$ is associated to $f$ and hence $\ad f$ is injective on $\g_e(1+rp)$ for each $r\geq 1$. It follows that $e_{\alpha}$ and $e_{\beta}$ can have no non-zero component in $\g_e(1+rp)$, in other words they are homogeneous in $\g_e(1)$.

Looking at $\g_e(1)$ in \cite{LT11}, we conclude that $e_{\alpha}$ and $e_\beta$ are both of the form {\def\arraystretch{0.5} \arraycolsep=0pt\begin{align*} \lambda_1\cdot e_{\Small\begin{array}{c c c c c c c c}0&0&0&0&1&1\\&&0\end{array}}+
&4\cdot\lambda_2\cdot e_{\Small\begin{array}{c c c c c c c c}1&1&1&1&0&0\\&&1\end{array}}+
3\cdot\lambda_2\cdot e_{\Small\begin{array}{c c c c c c c c}1&1&1&1&1&0\\&&0\end{array}}+
4\cdot\lambda_2\cdot e_{\Small\begin{array}{c c c c c c c c}0&1&2&1&0&0\\&&1\end{array}}+
\lambda_2\cdot e_{\Small\begin{array}{c c c c c c c c}0&1&1&1&1&0\\&&1\end{array}}+\lambda_3\cdot e_{\Small\begin{array}{c c c c c c c c}-0&0&0&0&0&1\\&&0\end{array}}
\\
&+2\cdot\lambda_4\cdot e_{\Small\begin{array}{c c c c c c c c}-1&1&1&1&0&0\\&&0\end{array}}+
4\cdot\lambda_4\cdot e_{\Small\begin{array}{c c c c c c c c}-0&1&1&1&0&0\\&&1\end{array}}+
4\cdot\lambda_4\cdot e_{\Small\begin{array}{c c c c c c c c}-0&0&1&1&1&0\\&&1\end{array}}+
\lambda_4\cdot e_{\Small\begin{array}{c c c c c c c c}-0&1&1&1&1&0\\&&0\end{array}},\end{align*}}with $\lambda_i\in k$. If $e_\alpha$ arises from the coefficients $(\lambda_1,\lambda_2,\lambda_3,\lambda_4)$ and $e_\beta$ from $(\mu_1,\mu_2,\mu_3,\mu_4)$ then calculating the commutator and insisting that the answer be $e$ one sees that the equations \[\lambda_1\mu_2=\lambda_2\mu_1,\quad \lambda_4\mu_3=\lambda_3\mu_4,\quad \lambda_4\mu_2-\lambda_2\mu_4=1,\quad \lambda_1\mu_3-\lambda_3\mu_1=1,\tag{*}\]
must be satisfied. If $\mu_1\neq 0$ then by replacing $e_\alpha$ by $e_\alpha-\nu e_\beta$ for suitable $\nu$, we may assume that $\lambda_1=0$. Otherwise we may swap $e_\alpha$ and $e_\beta$ to assume $\lambda_1=0$. Then using the equations of $(*)$ we have $\lambda_3\mu_1=-1$, thus $\mu_1\neq 0$ and so $\lambda_2=0$. Subsequently $\lambda_4\mu_2=1$. Now replacing $e_\alpha$ by a multiple we can arrange $\lambda_3=1$, thus $\mu_1=-1$. Additionally, (using \cite{LT11}) one checks that the element $\def\arraystretch{0.5} \arraycolsep=0pt h_{1}(t^6)h_{2}(t^9)h_{3}(t^{12})h_{4}(t^{18})h_5(t^{15})h_6(t^{10})h_7(t^5)$ centralises $e$ and the element of $\g_e(1)$ corresponding to coordinates $(0,0,1,0)$, while acting as a non-trivial scalar on $(0,0,0,1)$. It follows that we may replace $e_\alpha$ with a conjugate such that its coordinates are $(0,0,1,1)$, that is, $\lambda_1=\lambda_2=0$ and $\lambda_3=\lambda_4=1$. Thus $\mu_2=1$ and $\mu_3=\mu_4$. Replacing now $e_\beta$ by $e_\beta-\mu_3e_\alpha$, we may assume $\mu_3=\mu_4=0$ so that the coordinates of $e_\beta$ are $(-1,1,0,0)$. Hence $e_\alpha$ and $e_\beta$ are completely determined.

Now we show that $f_{\tilde\alpha}$ is unique up to conjugacy by $G_{e}\cap G_{h_{\tilde\alpha}}\cap G_{e_\alpha}\cap G_{e_\beta}$. Again if $(e,h_{\tilde\alpha},f)$ is a standard $\sl_2$-triple, then $f_{\tilde\alpha}-f\in\g_e(0)$ and as $h_{\tilde\alpha}$ has weight $-2$ on it, in fact, $f_{\tilde\alpha}-f\in\bigoplus_{r>0}\g_e(-2+rp)=\g_e(3)\oplus\g_e(8)$. Checking \cite{LT11}, we have $f_{\tilde\alpha}$ is of the form

{\def\arraystretch{0.5}\arraycolsep=0pt\begin{align*}\lambda_1\cdot &e_{\Small\begin{array}{c c c c c c c c}1&1&1&0&0&0\\&&1\end{array}}+
\lambda_2\cdot e_{\Small\begin{array}{c c c c c c c c}1&1&1&1&1&0\\&&1\end{array}}+
\lambda_2\cdot e_{\Small\begin{array}{c c c c c c c c}0&1&2&1&1&0\\&&1\end{array}}
-\lambda_3\cdot e_{\Small\begin{array}{c c c c c c c c}1&2&3&2&1&1\\&&2\end{array}}+
\lambda_3\cdot e_{\Small\begin{array}{c c c c c c c c}1&2&3&2&2&1\\&&1\end{array}}+
e_{\Small\begin{array}{c c c c c c c c}-1&0&0&0&0&0\\&&0\end{array}}
+2\cdot e_{\Small\begin{array}{c c c c c c c c}-0&0&0&0&0&0\\&&1\end{array}}\\
&+
2\cdot e_{\Small\begin{array}{c c c c c c c c}-0&0&1&0&0&0\\&&0\end{array}}+
e_{\Small\begin{array}{c c c c c c c c}-0&0&0&0&1&0\\&&0\end{array}}
-\lambda_4\cdot e_{\Small\begin{array}{c c c c c c c c}-0&0&1&1&0&0\\&&1\end{array}}+
\lambda_4\cdot e_{\Small\begin{array}{c c c c c c c c}-0&1&1&1&0&0\\&&0\end{array}}+
\lambda_5\cdot e_{\Small\begin{array}{c c c c c c c c}-1&1&2&2&1&1\\&&1\end{array}}+
\lambda_5\cdot e_{\Small\begin{array}{c c c c c c c c}-0&1&2&2&2&1\\&&1\end{array}}\end{align*}}
In $\h$ we have the relation $[[f_{\tilde\alpha},e_\alpha],e_\alpha]=0$. This implies $\lambda_2=\lambda_3=\lambda_4=0$. Additionally, the relation $[[f_{\tilde\alpha},e_\beta],e_\beta]=0$ implies $\lambda_5=0$. Lastly, $[[e_\alpha,f_{\tilde\alpha}],f_{\tilde\alpha}]=0$ implies $\lambda_1=0$. Thus $f_{\tilde\alpha}$ is fully determined.

We obtain in addition $e_{-\alpha}=[f_{\tilde\alpha},e_{\beta}]$ and $e_{-\beta}=[e_{\alpha},f_{\tilde\alpha}]$ giving in total,

{\def\arraystretch{0.5}\arraycolsep=0pt\begin{align*}e_\alpha&:=e_{\Small\begin{array}{c c c c c c c c}-0&0&0&0&0&1\\&&0\end{array}}+
2\cdot e_{\Small\begin{array}{c c c c c c c c}-1&1&1&1&0&0\\&&0\end{array}}+
4\cdot e_{\Small\begin{array}{c c c c c c c c}-0&1&1&1&0&0\\&&1\end{array}}+
4\cdot e_{\Small\begin{array}{c c c c c c c c}-0&0&1&1&1&0\\&&1\end{array}}+
e_{\Small\begin{array}{c c c c c c c c}-0&1&1&1&1&0\\&&0\end{array}}\\
e_\beta&:=-1\cdot e_{\Small\begin{array}{c c c c c c c c}0&0&0&0&1&1\\&&0\end{array}}+
4\cdot e_{\Small\begin{array}{c c c c c c c c}1&1&1&1&0&0\\&&1\end{array}}+
3\cdot e_{\Small\begin{array}{c c c c c c c c}1&1&1&1&1&0\\&&0\end{array}}+
4\cdot e_{\Small\begin{array}{c c c c c c c c}0&1&2&1&0&0\\&&1\end{array}}+
e_{\Small\begin{array}{c c c c c c c c}0&1&1&1&1&0\\&&1\end{array}}\\
e_{-\alpha}&:=e_{\Small\begin{array}{c c c c c c c c}0&0&0&0&0&1\\&&0\end{array}}+
2\cdot e_{\Small\begin{array}{c c c c c c c c}1&1&1&1&0&0\\&&0\end{array}}+
4\cdot e_{\Small\begin{array}{c c c c c c c c}0&1&1&1&0&0\\&&1\end{array}}+
4\cdot e_{\Small\begin{array}{c c c c c c c c}0&0&1&1&1&0\\&&1\end{array}}+
4\cdot e_{\Small\begin{array}{c c c c c c c c}0&1&1&1&1&0\\&&0\end{array}}\\
e_{-\beta}&:=e_{\Small\begin{array}{c c c c c c c c}-0&0&0&0&1&1\\&&0\end{array}}+
4\cdot e_{\Small\begin{array}{c c c c c c c c}-1&1&1&1&0&0\\&&1\end{array}}+
2\cdot e_{\Small\begin{array}{c c c c c c c c}-1&1&1&1&1&0\\&&0\end{array}}+
e_{\Small\begin{array}{c c c c c c c c}-0&1&2&1&0&0\\&&1\end{array}}+
4\cdot e_{\Small\begin{array}{c c c c c c c c}-0&1&1&1&1&0\\&&1\end{array}}.
\end{align*}}
It is automatic from the Serre relations that these elements generate a subalgebra isomorphic to $A_2$. However if one chases through the proof of \cite[Lem.~4.1.3]{LS04}, we discover that the conjugacy class of maximal $A_2$ subgroups of $E_7$ have Lie algebras whose root elements are of type $A_4+A_1$. It follows that $\h=\Lie(H)$ for one of these subgroups.  
\end{proof}

\begin{lemma}\label{mtfora2}Theorem \ref{maintheorem} holds when $\h$ is of type $A_2$.\end{lemma}
\begin{proof}Suppose, for a contradiction, that $\h$ is a non-$G$-cr subalgebra of type $A_2$. By (\ref{VinTable}) we have that there is an $\bar\h$-composition factor $V$ of $\q$ with $V = L(\lambda)$ or $L(\lambda)^*$ where $\lambda$ is $(3,1)$ or $(3,3)$ when $p=5$, or $(5,1)$ when $p=7$. Furthermore, by Lemma \ref{nobadthingsinclassicalLevis} we may assume that $\l'$ is of type $E_6$ or $E_7$ and  that $\bar{\h}$ is a $p$-subalgebra by Corollary \ref{kwcor}.

Suppose $p=5$. Since $\dim L(3,3) = 63$, it cannot occur as an $\bar\h$-composition factor of $\q$. We will require more work to show that $L(3,1)$ or $L(1,3)$ cannot occur as an $\bar\h$-composition factor of $\q$. Since the argument is the same for each, we assume $V=L(3,1)$.

By Lemma \ref{reachlem}, any root vector of $\bar \h$, say $e_{\tilde\alpha}$ corresponding to the highest root $\tilde \alpha$, has the property that it is reachable. Then the possibilities are given by \cite{PreSteNilp}. Let also $h_{\tilde\alpha}$ be an element in the Cartan subalgebra of $\bar \h$ for which there is a nilpotent element $f_{\tilde\alpha}$ with $\s_{\tilde\alpha}=\la e_{\tilde\alpha},h_{\tilde\alpha},f_{\tilde\alpha}\ra\cong\sl_2$. 

We need certain data about the restrictions to $\s_{\tilde\alpha}$ of all irreducible restricted non-self-dual $\bar\h$-modules of dimension at most $10$ and all self-dual irreducible $\bar\h$-modules of dimension at most $20$.
\begin{table}\begin{center}\begin{tabular}{llll}
$\lambda$ & $\dim L(\lambda)$ & $L(\lambda)\downarrow\s_{\tilde\alpha}$ & weights\\\hline
$(2,2)$ & $19$ & $L(4)+L(3)^2+L(2)^2$ & $4^3/3^5/2^5/1^3/0^3$\\
$(3,1)$ & $18$ & $L(4)+L(3)^2+L(2)+L(1)$ & $4^4/3^4/2^4/1^4/0^2$\\
$(3,0)$ & $10$ & $L(3)+L(2)+L(1)+k$ & $4^2/3^2/2^2/1^2/0^2$\\
$(1,1)$ & $8$ & $L(2)+L(1)^2+k$ & $4^2/3/2/1^2/0^2$\\
$(2,0)$ & $6$ & $L(2)+L(1)+k$ & $4/3/2/1/0^2$\\
$(1,0)$ & $3$ & $L(1)+k$ & $4/1/0$
\end{tabular}\end{center}\caption{The restrictions of various $\sl_3$-modules}\label{a2resttab}\end{table}
We also note that $e_{\tilde\alpha}$ has a Jordan block of size $5$ on $V$.

Suppose first that $\l'$ is of type $E_6$. Then $V$ occurs as a $\bar\h$-composition factor of $V_{27}$ or its dual. One finds that in $E_6$ there is just one reachable element with a Jordan block of size at least $5$ on $V_{27}$, namely that with label $2A_2+A_1$.  Using Table~\ref{a2resttab}, we will now compare the weights of $h_{\tilde\alpha}$ on $V_{27}$ with those on various modules of dimension $27$ containing $V$ as a composition factor. It will be convenient to do this by considering the action of elements of $E_6$ on elements of $\q/[\q,\q]$ where $\q$ is the nilradical of an $E_6$ parabolic subalgebra of $E_7$. The tables in \cite{LT11} give a cocharacter $\tau$ associated to $2A_2+A_1$ in $E_7$. As $e_{\tilde\alpha}$ does not contain a factor of type $A_{p-1}$ then by Proposition \ref{findingcoch}, we may assume $h_{\tilde\alpha}\in\Lie(\tau(\Gm))$ hence it suffices to compute the weights of $\tau$ on $\q/[\q,\q]$, which is an easy computation in the root system of $E_7$ (and which we perform in GAP). We find the weights of $\tau$ on $V_{27}$ are $-4/-3^2/-2^4/-1^4/0^5/1^4/2^4/3^2/4$. Reduction modulo $5$ then gives the weights of $h_{\tilde\alpha}$ on $V_{27}$, namely $4^5/3^6/2^6/1^5/0^5$. Since one composition factor is $V=L(3,1)$ itself, we must find $4/3^2/2^2/1^3/0^3$ from the remaining composition factors of $V_{27}$. It is then a straightforward matter to compare this with the weights of the tables above to see that this is impossible.

Thus we must have $\l'$ of type $E_7$. Then $V$ and its dual occur on the self-dual module $V_{56}\downarrow\bar\h$. It follows that a subset of the composition factors of $V_{56}\downarrow \s_{\tilde\alpha}$ is $L(4)^2 / L(3)^4 / L(2)^2 / L(1)^2$. Hence a root element $e_{\tilde\alpha}$ of $\s_{\tilde\alpha}$ is represented on $V_{56}$ by a matrix whose $4$th power has rank at least two and $3$rd power has rank at least $6$. Since it is also reachable, by comparison with \cite{PreSteNilp} there are just three options for the nilpotent orbit containing $e_{\tilde\alpha}$, namely $A_3+A_2+A_1$, $2A_2+A_1$ and $A_4+A_1$. The case $A_3+A_2+A_1$ can be ruled out: in both $\sl_3$ and $\g$ there is a unique toral element $h\in \im\ad e_{\tilde\alpha}$ which has weight $2$ on a highest root element, and has weight $1$ on the simple root elements. As $\g_{e_{\tilde\alpha}}\subseteq \g_{e_{\tilde\alpha}}(\geq 0)$, it follows that $\g_{e_{\tilde\alpha}}(1)$ must be non-zero. But \cite{LT11} reveals that for $e_{\tilde\alpha}$ in $E_7$ of type $A_3+A_2+A_1$ the space $\g_{e_{\tilde\alpha}}(1)$ is zero. (The element $e_{\tilde\alpha}$ can in fact be found in $[\g_{e_{\tilde\alpha}}(0),\g_{e_{\tilde\alpha}}(2)]$.) This is a contradiction. Hence $e_{\tilde\alpha}$ is not of type $A_3+A_2+A_1$.

Suppose $e_{\tilde\alpha}$ is of type $2A_2+A_1$. Since $e_{\tilde\alpha}$ is not regular in a Levi with a factor of type $A_{p-1}$, Proposition \ref{findingcoch} implies that there is a unique $h_{\tilde\alpha}\in\h$ such that $h_{\tilde\alpha}$ has weight $2$ on $e_{\tilde\alpha}$ and $h_{\tilde\alpha} \in \im\ad e_{\tilde\alpha}$; we take $h_{\tilde\alpha}\in\Lie(\tau(\Gm))$ for $\tau$ an associated cocharacter to $e_{\tilde\alpha}$. The precise description of such a $\tau$ comes from \cite{LT11}. In an $E_7$-parabolic subalgebra of $E_8$,  the space $\q/[\q,\q]$ affords a representation $V_{56}$ for the Levi. Applying $\tau$ to the roots of $\q/[\q,\q]$ gives the following multiplicities for $\tau$:

\begin{center}\begin{tabular}{l|lllllllll}
wt & -4&-3&-2&-1&0&1&2&3&4\\\hline
$\dim$ & 2&4&8&8&12&8&8&4&2\end{tabular}\end{center}

Since $\ad e_{\tilde\alpha}$ kills the $L(4)$ weight space, on which $h_{\tilde\alpha}$ has weight $4$, we must have two composition factors isomorphic to $L(4)$. These afford weights $(4,2,0,-2,-4)$ on $\q/[\q,\q]$. Removing the weights of these composition factors and considering the remaining weights, we see inductively that the composition factors of $\s_{\tilde\alpha}$ are $L(4)^2/L(3)^4/L(2)^6/L(1)^4/k^4$. 

Both $V$ and its dual occur as composition factors of $\bar \h$ on $V_{56}$, so using Table~\ref{a2resttab}, we find that the restriction to $\s_{\tilde\alpha}$ of the remaining $\bar \h$-composition factors of $V_{56}$ has composition factors $L(2)^4 / L(1)^2 / k^4$. Again using Table~\ref{a2resttab}, we find that either $L(1,1)$ occurs or both $L(2,0)$ and $L(0,2)$ occur as $\bar \h$-composition factors of $V_{56}$. In the first case, the remaining $\s_{\tilde\alpha}$-composition factors are $L(2)^3 / k^3$. But no combination of modules in Table~\ref{a2resttab} have such restriction, a contradiction. So $L(2,0)$ and $L(0,2)$ both occur, and the remaining $\s_{\tilde\alpha}$-composition factors are $L(2)^2 / k^2$. Again we see that this is impossible. This rules out the case $e_{\tilde\alpha}$ of type $2A_2+A_1$.  

If $e_{\tilde\alpha}$ is of type $A_4+A_1$, then Lemma \ref{a2max} implies that $\h=\Lie(H)$ for $H$ a maximal $A_2$ subgroup of $G$. 
However from \cite[Lem.~4.1.3]{LS04} we have found the Lie algebra of the maximal connected subgroup of type $A_2$ in $E_7$. But this does not act on $V_{56}$ with a composition factor of high weight $(3,1)$ or $(1,3)$. This contradiction completes this case.

Finally, we consider the possibility that $V=L(5,1)$ when $p=7$. Since $\dim L(5,1) = 33>27$, we cannot have $\l'$ of type $E_6$; also the self-duality of $V_{56}$ implies that we would need $L(5,1)$ and $L(1,5)$ as composition factors of $V_{56}$, but $66>56$. This is a contradiction.
\end{proof}

\begin{lemma}\label{mtforb2}Theorem \ref{maintheorem} holds when $\h$ is of type $B_2$.\end{lemma}
\begin{proof}Assume $\h$ is of type $B_2$ and non-$G$-cr. By (\ref{VinTable}) we have at least one of $L(2,0)$ and $L(1,3)$ occurs as an $\bar\h$-composition factor of $\q$ when $p=5$ or $L(4,0)$ occurs as an $\bar\h$-composition factor of $\q$ when $p=7$. Moreover, by Lemma \ref{nobadthingsinclassicalLevis} we have $\g$ is of type $E_7$ or $E_8$.  

First suppose $p=5$. When $\g$ is of type $E_8$ we construct an example of a non-$G$-cr subalgebra of type $B_2$. Let $\m$ be a maximal subsystem subalgebra of type $D_8$. We embed a subalgebra $\h$ of type $B_2$ into $\m$ via the representation $T(2,0) + k = k|L(2,0)|k + k$. (Note that $T(2,0)$ is self-dual and odd-dimensional, hence is an orthogonal representation for $\h$.) Therefore $\h$ is contained in a $D_7$-parabolic subalgebra of $\m$: by Lemma \ref{girclassical}, it is in some parabolic, and it stabilises a singular vector, the stabiliser of which is a $D_7$-parabolic. Then $\h$ is thus in a $D_7$-parabolic of $\g$. We know that a Levi subalgebra of type $D_7$ acts as $L(\varpi_1)^2 + L(\varpi_2) + L(\varpi_6) + L(\varpi_7) +0$ on $\g$. In particular, the largest summand is $91$-dimensional. Now consider $H$, a subgroup $B_2$ of $G =  E_8$ embedded into $D_8$ via $T(2,0) + k$. Then $\bigwedge^2(T(2,0) + k)$ occurs as direct summand of $\g \downarrow H$. For any $p>2$, we have that $\bigwedge^2 W$ is a direct summand of $W^{\otimes 2}$, hence if $W$ is tilting, $\bigwedge^2 W $ is also. In our case, this implies  $\bigwedge^2(T(2,0) + k) \cong T(2,0) + L(0,2) + T(2,2)$. Therefore, $H$ has a $95$-dimensional indecomposable summand $M\cong T(2,2)\cong L(2,0)|L(2,2)|L(2,0)$ on $\g$. We may identify $T(2,2)$ with the $H_1$-injective hull of $L(2,0)$, which restricts to $H_1$ indecomposably. As the category of representations for $H_1$ is equivalent to the category of $p$-representations for $\h$, this shows $\h$ has an indecomposable summand $M\downarrow \h$ of dimension at least $95$ on $\g$; thus $\h$ cannot live in a Levi subalgebra of type $D_7$, proving it is non-$G$-cr. 

Now suppose $\g$ is of type $E_7$. Then by Lemma \ref{nobadthingsinclassicalLevis}, we have $\l'$ is of type $E_6$ and $V_{27} \downarrow \bar\h$ has a composition factor $L(2,0)$ (we rule out $L(1,3)$ since it is $52$-dimensional). Using \cite[6.22]{Lub01}, the only irreducible $\bar \h$-modules of dimension at most $14$ are $L(2,0)$, $L(1,1)$, $L(0,2)$, $L(1,0)$, $L(0,1)$ and $k$. Let $\s$ denote a long root $\sl_2$-subalgebra of $\bar\h$. The following table lists the composition factors of the restrictions of the irreducible $\bar\h$-modules above to $\s$.
\begin{center}\begin{tabular}{l l l}$\lambda$ & $\dim L(\lambda)$ & $L(\lambda)\downarrow\s$\\\hline
$(0,0)$ & $1$ & $k$\\
$(0,1)$ & $4$ & $L(1)/k^2$\\
$(1,0)$ & $5$ & $L(1)^2/k$\\
$(0,2)$ & $10$ & $L(2)/L(1)^2/k^3$\\
$(1,1)$ & $12$ & $L(2)^2/L(1)^3$\\
$(2,0)$ & $13$ & $L(2)^3/L(1)^2$\end{tabular}\end{center}
A non-zero nilpotent element $e$ of $\s$ satisfies $e\in[\bar\h_e,\bar\h_e]$ i.e. it is reachable in $\g$. Thus from \cite{PreSteNilp} it is of type $A_1$, $2A_1$, $3A_1$, $A_2+A_1$, $A_2+2A_1$ or $2A_2+A_1$. As in Lemma \ref{mtfora2} we establish that the composition factors of $\s$ on $V_{27}$ must be as follows:
\begin{center}\begin{tabular}{l l}$\OO$ & $V_{27}\downarrow\s$\\\hline
$A_1$ & $L(1)^6/k^{15}$\\
$2A_1$ & $L(2)/L(1)^8/k^8$\\
$3A_1$ & $L(2)^3/L(1)^6/k^6$\\
$A_2+A_1$ & $L(3)/L(2)^4/L(1)^4/k^3$\\
$A_2+2A_1$ & $L(3)^2/L(2)^3/L(1)^4/k^2$\\
$2A_2+A_1$ & $L(4)/L(3)^2/L(2)^3/L(1)^2/k$\end{tabular}\end{center}
Comparing the above two tables, we find that there is just one possibility: $\OO$ is of type $3A_1$ and the $\bar\h$-composition factors of $V_{27}$ are $L(2,0)/L(1,0)^2/k^4$. Let $\la e,h,f\ra=\s\subset \bar\h$ be an $\sl_2$-triple. By Proposition \ref{findingcoch}, up to conjugacy by $G_e$ we have that $h\in\Lie(\tau(\Gm))$ for $\tau$ an associated cocharacter to $e$. Up to conjugacy then, we may assume $e=e_{\alpha_1}+e_{\alpha_3}+e_{\alpha_6}$, $h=h_{\alpha_1}+h_{\alpha_3}+h_{\alpha_6}$ and there is a nilpotent element $e'$ in an $\sl_2$-subalgebra $\s'$ of $\bar\h$ which commutes with $e$ and $h$. The subalgebra $\s'$ must also be a long root $\sl_2$-subalgebra of $\bar\h$, and so $e'$ is also reachable of type $A_1^3$. Now, from \cite{LT11} we see that $\g_e(0)=C_e$ for $C_e$ a reductive group of type $A_2A_1$. We have that $C_e$ has six nilpotent orbits on $\g_e(0)$ corresponding to partitions of $(3,2)$; {\it viz.} $\{3,2+1,1+1+1\}\times\{2,1+1\}$. These can be computed in GAP as having orbit types in $E_6$ with labels $2 A_2 A_1$, $2 A_2$, $3 A_1$, $2 A_1$, $A_1$, $0$. Up to conjugacy by $C_e$ then, there is just one possibility for $e'$, which may be taken as $e_{\alpha_4+\alpha_5}+e_{\alpha_5+\alpha_6}$. The element $e_s:=e+e'$ is then a subregular nilpotent element of $\bar\h$, which in $\g$ is checked to have type $A_2+2A_1$. A corresponding $\sl_2$ subalgebra, $\s_s$ say has composition factors on $L(2,0)$ which are $L(4)/L(2)^2/k^4$. But the existence of the $L(4)$ composition factor is already incompatible with the action of an $\s_s$ on $V_{27}$ from the table above. This is a contradiction. 

Now suppose $p=7$. Then by Proposition \ref{h1sless64}, we have $L(4,0)$ occurring as a $\bar\h$-composition factor of $\q$. Thus the $\bar\h$-composition factors on $V_{56}$ are $L(4,0)/k^2$. One computes that the restriction of $V_{56}$ to a Levi subalgebra $\sl_2$ is completely reducible with composition factors $L(4)^5/L(3)^4/L(2)^3/L(1)^2/k^3$. This action is inconsistent with any of the Jordan blocks in \cite[Tables~2,~3]{SteComp}.\end{proof}

\begin{lemma}\label{mtforg2}Theorem \ref{maintheorem} holds when $\h$ is of type $G_2$.\end{lemma}
\begin{proof}
Assume that $\h$ is a non-$G$-cr subalgebra of type $G_2$. By (\ref{VinTable}), we have $p=7$ and $\g$ is of type $E_7$ or $E_8$. In both cases, we will construct a non-$G$-cr subalgebra of type $G_2$.  

Let $\g$ be of type $E_7$ and $p=7$. Then $\g$ contains a parabolic subalgebra $\p = \l + \q$, with $\l'$ of type $E_6$ and the $\l'$-composition factor of $q$ is $V_{27}$. Now, $\l'$ contains a maximal subalgebra $\bar{\h}$ of type $G_2$ and $V_{27} \downarrow \bar{\h} = L(2,0)$. By Proposition \ref{h1sless64}, we have $\opH^1(\bar{\h}, \q) \cong k$. Thus, by Lemma \ref{hnongcr}, there must exist a non-$G$-cr subalgebra $\h$ of type $G_2$. 

Now let $\g$ be of type $E_8$ and $p=7$. Then $\g$ has a Levi factor $\l'$ of type $E_7$ and by Lemma \ref{levinongcr}, the non-$L'$-cr subalgebra of $E_7$ constructed above is therefore a non-$G$-cr subalgebra of $\g$.  
\end{proof}

\subsection{Subalgebras of type $W_1$} \label{sec:w1s}

 Our strategy in this section is slightly different. When $p = 7$ and $\g$ is of type $F_4$ or $p=5,7$ and $\g$ is of type $E_6$, $E_7$ or $E_8$ we construct an example of a non-$G$-cr subalgebra of type $W_1$. In all other cases, we use calculations in GAP to show that each subalgebra of type $W_1$ is $G$-cr. To do this we rely on the following result:

\begin{lemma}[cf. {\cite[Thm.~1.1]{HSMax}}] \label{W1s}
Let $\g$ be a simple classical Lie algebra of exceptional type. Suppose $\h \cong W_1$ is a $p$-subalgebra of $\g$ and $p$ is a good prime for $\g$. Let $\del \in \h$ be represented by the nilpotent element $e \in \g$. Then the following hold: 
\begin{enumerate}
\item $e$ is a regular element in a Levi subalgebra $\l$ of $\g$ and the root system associated to $\l$ is irreducible. 
\item For $h(L)$ the Coxeter number of $\l$, we have either $p = h(L)+1$ or $\l$ is of type $A_n$ and $p=h(L)$. 
\end{enumerate}
\end{lemma}

This will allow us to construct in GAP a generic subalgebra $\h$ of type $W_1$ for a representative of each possible nilpotent element representing $\del$ and then show that $X^{p-1} \del$ is also contained in $\l$, hence $\h = \la \del, X^{p-1} \del \ra$ is contained in $\l$.   

\begin{lemma}\label{mtforw1}Theorem \ref{maintheorem} holds when $\h$ is of type $W_1$.\end{lemma}
\begin{proof}
By Corollary \ref{kwcor}
, we may assume that $\h$ is non-$G$-cr and the projection of $\h$ to $\l'$ is a $p$-subalgebra unless $p=5,7$ and $\g$ has a Levi subalgebra of type $A_{p-1}$.

When $\g$ is of type $G_2$ all Levi factors are of type $A_1$, so this case is immediately discounted. 

Now suppose $\g$ is of type $F_4$. When $p=5$ we show that all subalgebras of type $W_1$ are $G$-cr but we postpone doing this here and give a general method below. When $p=7$ we claim the following  subalgebra is non-$G$-cr: \[\h = \la e_{0100}+e_{0010}+e_{0001},e_{-0122}+e_{-1222}+4\cdot e_{-1231}\ra.\] By checking the commutator relations hold, we see that $\h$ is isomorphic to $W_1$ (with the first generator mapped to $\del$ and the second mapped to $X^{p-1} \del$). Moreover, $\h$ is evidently contained in a $C_3$-parabolic subalgebra. Now, using the MeatAxe in GAP, we calculate that the socle of the adjoint module $\g \downarrow \h$ is $15$-dimensional. On the other hand, any subalgebra of type $W_1$ contained in a Levi subalgebra of type $C_3$ is conjugate to $\bar\h = \la e_{0100}+e_{0010}+e_{0001},e_{-0122} \ra$ but using the MeatAxe, we calculate that the socle of $\g \downarrow \bar\h$ is $24$-dimensional. Therefore $\h$ is not contained in a Levi subalgebra of $C_3$ and is thus non-$G$-cr. When $p \geq 11$, all subalgebras of $\g$ of type $W_1$ are $G$-cr by Lemma \ref{W1s}, since the largest Coxeter number of a proper Levi subalgebra of $\g$ is $6$.

Now let $\g$ be of type $E_6$. We construct an example of a non-$G$-cr subalgebra of type $W_1$ when $p=5,7$. For $p=5$ the subalgebra $\h \cong W_1$ embedded in a Levi subalgebra of type $A_4$ via the representation $k[X]/X^p\cong L(4) | k$ is non-$A_4$-cr and hence non-$G$-cr by Lemma \ref{levinongcr}. For $p=7$, consider the following subalgebra.  
\def\arraystretch{0.5} \arraycolsep=0pt 
\begin{align*}\h=\langle e_{\footnotesize\begin{array}{c c c c c}1&0&0&0&0\\&&0\end{array}}+
e_{\footnotesize\begin{array}{c c c c c}0&1&0&0&0\\&&0\end{array}}+
e_{\footnotesize\begin{array}{c c c c c}0&0&1&0&0\\&&0\end{array}}+
e_{\footnotesize\begin{array}{c c c c c}0&0&0&1&0\\&&0\end{array}}+
e_{\footnotesize\begin{array}{c c c c c}0&0&0&0&1\\&&0\end{array}},\\
e_{\footnotesize\begin{array}{c c c c c}- 1&1&1&1&1\\&&0\end{array}}-
2\cdot e_{\footnotesize\begin{array}{c c c c c}- 1&1&2&1&1\\&&1\end{array}}+
e_{\footnotesize\begin{array}{c c c c c}- 1&2&2&1&0\\&&1\end{array}}+
e_{\footnotesize\begin{array}{c c c c c}- 0&1&2&2&1\\&&1\end{array}}\ra
\end{align*}
Again, one checks that $\h$ is isomorphic to $W_1$ and is evidently contained in an $A_5$-parabolic subalgebra. We then use the MeatAxe to calculate that the socle of $\g \downarrow \h$ is $21$-dimensional, whereas any subalgebra of type $W_1$ contained in a Levi subalgebra of type $A_5$ acts on $\g$ with a $43$-dimensional socle. Therefore $\h$ is non-$G$-cr. When $p \geq 11$, all subalgebras of type $W_1$ are $G$-cr by Lemma \ref{W1s}, since the largest Coxeter number of a proper Levi subalgebra of $\g$ is $8$.

Finally, suppose $\g$ is of type $E_7$ or $E_8$. Both contain an $E_6$-Levi subalgebra and therefore contain a non-$G$-cr subalgebra of type $W_1$ when $p=5,7$ by Lemma \ref{levinongcr}. We now consider the case $p \geq 11$. Therefore we have that $\bar\h$ is a $p$-subalgebra. By Lemma \ref{W1s}, it follows that $(\g,\l',p) = (E_7,D_6,11)$, $(E_7,E_6,13)$, $(E_8,D_6,11)$, $(E_8,E_6,13)$, $(E_8,D_7,13)$ or $(E_8,E_7,19)$ and that $\del=e$ is regular in $\l$. We rule out each possibility, as well as $(F_4,B_2,5)$, using calculation in GAP. All of the cases are similar and we give the general method.  

Let $e$ be the regular nilpotent element of $\l$. Then following the proof of \cite[Lem.~3.11]{HSMax} we have an associated cocharacter $\tau$ with $\Lie(\tau(\Gm)) = \la X\del \ra$. This cocharacter is explicitly given in \cite{LT11}. Now suppose $X^{p-1} \del$ is represented by the nilpotent element $f$. Then as $[X\del,X^{p-1} \del]=(p-1)X^{p-1} \del$, one calculates that $f$ is in the direct sum of the $\tau$ weight spaces congruent to $-2p+4$ modulo $p$. As in \cite[\S A.2]{HSMax} we use GAP to construct a generic nilpotent element $f_1$ in such weight spaces. Using the commutator relations in $W_1$, for example $\ad(e)^{p-1} f = -e$ and $[f,\ad(e)^i(f)] = 0$ for all $1\leq i\leq p-3$, we then find that $f_1$, and hence $f$, is contained in $\l$ and thus $\h = \la e, f \ra$ is contained in $\l$. Thus all subalgebras of type $W_1$ are $G$-cr in each possibility.   
\end{proof}

\subsection{Proofs of Theorems \ref{ssgcrcorollary} and \ref{maintheorem}.}
Of importance to us will be the following theorem:

\begin{theorem}[{\cite[Thm.~1.3]{HSMax}}]\label{hsw1orclassical}Let $\g$ be a simple classical Lie algebra of exceptional type.
Suppose $p$ is a good prime for $\g$ and let $\h$ be a simple subalgebra of $\g$. Then $\h$ is either isomorphic to $W_1$ or it is of classical type.
\end{theorem}

\begin{proof}[Proof of Theorem \ref{maintheorem}]
As $\h$ must project to an isomorphic subalgebra of a proper Levi subalgebra in good characteristic, the theorem now follows from Theorem \ref{hsw1orclassical} and Lemmas \ref{rankatleast3}, \ref{mtfora1a}, \ref{mtfora1b}, \ref{mtfora2}, \ref{mtforb2}, \ref{mtforg2} and \ref{mtforw1} above.
\end{proof}

\begin{proof}[Proof of Theorem \ref{ssgcrcorollary}]
Suppose $G$ is connected reductive with $\g$ its Lie algebra and $\h$ some semisimple subalgebra. Lemma \ref{a1nongcr} provides the forward implication and so it remains to prove the reverse one. Since we are assuming that $p>h(G)$, it is in particular a very good prime, and so we have $\g\cong\g_1\times\g_2\times\dots\times\g_r\times \z$ where each $\g_i$ is simple and $\z$ is a central torus of $\g$. The parabolic subalgebras of $\g$ are the direct products of parabolic subalgebras of the simple factors, and similarly for the corresponding Levi factors. Hence if $\h$ is in a parabolic subalgebra $\p$ of $\g$, then it is in a Levi subalgebra $\l$ of $\p$ if and only if the projection of $\h$ to each simple factor $\g_i$ of $\g$ also has this property. Thus we reduce the proof of the theorem to the case $G$ is simple. Now if $G$ is classical, the result is supplied by Proposition \ref{thm:gcrclassical}. Thus we may assume that $G$ is exceptional. It will be shown in a forthcoming paper \cite{PSMax} by A.~Premet and the first author, that all semisimple subalgebras are direct sums of simple Lie algebras when $p > h(G)$. Putting this together with Theorem \ref{hsw1orclassical} and Lemma \ref{W1s}, we have that $\h=\h_1\times\cdots\times\h_r$ for each $\h_i$ a simple classical Lie algebra. 

Assume $\h$ is a subalgebra of a parabolic subalgebra $\p=\l+\q$ of $\g$. We will be done by Proposition \ref{h1ofqis0} if we can show that any simple $\h$-composition factor $V\cong V_1\otimes \dots\otimes V_r$ with $V_i$ a simple $\h_i$-module satisfies $\opH^1(\h,V)=0$. By the K\"unneth formula, we are done if we can show that $\opH^1(\h_i,V_i)=0$ for each $i$. This has already been shown to be impossible in the proof of Theorem \ref{maintheorem}: in the context of statement (\ref{VinTable}) we assumed the existence of such a module $V_i$, the contradiction of which proved Theorem \ref{maintheorem}.\end{proof}

\section{Unique embeddings of nilpotent elements into $\sl_2$-subalgebras. Proofs of Theorems \ref{jmthm}  and \ref{gooda1s}}\label{sec:jmthmprf}

\begin{proof}[Proof of Theorem \ref{jmthm}]
First suppose $p>h(G)$. We wish to show that the bijection (*) from the introduction holds.  To start with, \cite{Pom80} provides the surjectivity. It remains to prove that the map is injective.

Let $(e,h,f)$ be an $\sl_2$-triple of $\g$. By Theorem \ref{ssgcrcorollary} we have that $\h=\la e,h,f\ra$ is $G$-cr and by Lemma \ref{psubalgebraornongcr}, we have that $\h$ is a $p$-subalgebra. In particular, the element $e$ is nilpotent with $e^{[p]}=0$. We will show, under our hypotheses, that $\h$ is $L$-irreducible in a Levi subalgebra $\l=\Lie(L)$ of $G$ if and only if the element $e$ of $\h$ is distinguished in $\l$. One way round is easy: If $\h$ is contained in a Levi subalgebra $\l$ and $e$ is a distinguished element, then $\h$ cannot be in a proper parabolic subalgebra of $\l$, since if it did, then by Theorem \ref{ssgcrcorollary}, $\h$, hence also $e$, would be in a proper Levi subalgebra of $\l$. This is a contradiction as $e$ is assumed distinguished.

For the other direction, assume $\h$ is $L$-irreducible in some Levi subalgebra $\l=\Lie(L)$ of $\g$ and assume, looking for a contradiction, that $e$ is not distinguished. Let us see that this implies $\h$ is in a proper Levi subalgebra of $\l$. To do this, note first that $\h$ is a subalgebra of the Lie algebra $\l'=[\l,\l]$; we have $\l'$ is semisimple, due to our hypothesis that $p>h(G)$. We wish to show that $\h$ centralises a vector $w\in\l'$, since then $\h\subseteq(\l')_w$ will be in a proper parabolic subalgebra of $\l'$, hence by Theorem \ref{ssgcrcorollary}, in a proper Levi subalgebra of $\l'$. To see that $\h$ does indeed fix a vector, let us start by noting that $e$ has at least one Jordan block of size $1$ on the adjoint module of $\l'$. And for this, let $e\in\k\subseteq\l$ be a Levi subalgebra of $\l$ in which $e$ is distinguished and note that $p>h(G)$ implies that $\k$ contains no factors of type $A_{p-1}$ so that $\k=\k'\oplus\z(\k)$. That $\z(\k)\neq 0$ provides the existence of the requisite Jordan block. Moreover, \cite[Prop.~3.3]{HSMax} implies that $h\in\k$ also, so that $\z(\k)$ is also centralised by $h$. Let $0\neq v\in\z(\k)$. Then $[e,v]=0$ and $v\not\in\im \ad e$. If $[f,v]=0$ then $\h$ centralises $v$, so we are done. Otherwise, consider the $\h$-submodule $W:=\la v, \ad(f)v,\dots, \ad(f)^{p-1}v\ra$. (This is a submodule since each $\ad(f)^i(v)$ is an $\ad h$-eigenvector and so $W$ is $\ad e$-stable; additionally it is $\ad(f)$ stable since the fact that $\h$ is a $p$-subalgebra implies that $\ad(f)^p=0$.) Since $W':=\la \ad(f)v,\dots, \ad(f)^{p-1}v\ra$ is both $\ad e$- and $\ad f$-stable, we have that the $\h$-submodule $W$ is a non-trivial extension of $W'$ by $k$. But $\dim W'\leq p-1$ and the only simple $\h$-module which extends the trivial is the module $L(p-2)$ of dimension $p-1$. It follows that $W\cong k|L(p-2)$. Let $\widetilde{W}$ be an indecomposable summand of $\l$ containing $W$ as a submodule. We cannot have $\widetilde{W}$ projective since then its restriction to the subalgebra $ke\subset\h$ would give Jordan blocks of size $p$, which is not possible by the choice of $v$. Hence $\widetilde{W}$ is indecomposable and reducible. The structure of such modules was determined in \cite{Pre91} (see \cite[\S4.1]{Far09} for a more recent account): they have Loewy length $2$ with isotypic socle and head. Thus the head of $\widetilde{W}$ consists of trivial modules. But then the socle of $\widetilde{W}^*\subseteq (\l')^*\cong (\l')$ consists of trivial submodules. This implies that $\h$ fixes a $1$-space on $\l'$. This justifies the claim that $\h$ is in a proper Levi subalgebra of $\l$, which is the contradiction sought.

We now wish to show that if $\h$ is an $L$-irreducible subalgebra containing the nilpotent element $e$ distinguished in $\l$, then it is unique up to conjugacy in $L$. For this, recall that for any nilpotent element $e$ there is, by \cite{Pre95}, an associated cocharacter $\tau:\Gm\to L$ which gives a grading $\l=\bigoplus_{i\in\Z}\l(\tau;i)$. Moreover, the images of any two such cocharacters are conjugate by $L_e$. By Proposition \ref{findingcoch}, we may assume that if $(e,h,f)$ is an $\sl_2$-triple, that the element $h$ is contained in $\Lie(\tau(\Gm))$ and thus is unique up to conjugacy by $L_e$, contained in the graded piece $\l(\tau;0)$. Now suppose $e$ is distinguished and $(e,h,f)$ and $(e,h,f')$ are two $\sl_2$-triples. Then $f-f'\in\l_e=\l_e(\geq 0)$. But the weight of $h$ on $f$ and $f'$ is $-2$, so that $f-f'$ is an element of $\sum_{i>0}\l_e(-2+ip)$. Since $p>h(G)$ and the largest $j$ such that $\l(j)\neq 0$ is $2h-2$ (this follows from \cite[Prop.~30]{McN05}), we have that $\sum_{i>0}\l_e(-2+ip)=\l_e(-2+p)$. But since $e$ is distinguished in $\l$ we have that $\l_e(i)=0$ for all odd $i$, hence that $f-f'=0$ as required. This proves Theorem \ref{jmthm} for $p>h(G)$.

For $p\leq h(G)$ we appeal to Lemma \ref{a1nongcr}. If $(e,h,f)$ is an $\sl_2$-triple as described in the statement of the lemma, then $f$ is a regular nilpotent in a proper Levi subalgebra $\l$, say. Thus by \cite{Pom80} (or just another application of Lemma \ref{a1nongcr}) it can be embedded into an $\sl_2$-triple $(e',h',f)$ inside $\l$. Since $(e,h,f)$ is not contained in $\l$ we have that $(f,-h,e)$ and $(f,-h',e')$ are non-conjugate $\sl_2$-triples containing the common nilpotent element $f$ as required.\end{proof}

\begin{proof}[Proof of Theorem \ref{gooda1s}]
We wish to see that $\h=\Lie(H)$ for $H$ a good $A_1$-subgroup of $G$. Fix a pair $e\in\h$. The assumption $p>h(G)$ implies that all unipotent elements of $G$ are of order $p$. Let $u$ be one corresponding to $e$ under a Springer isomorphism. Now \cite[Props.~4.1 \& 4.2]{Sei00} furnish us with a good $A_1$-overgroup of any unipotent element $u$. Since all unipotent elements of $H$ are conjugate, the Lie algebra of a root group of $H$ will contain $e$. Thus $e\in\h=\Lie(H)$ as required.
\end{proof}

\section{Complete reducibility of $p$-subalgebras and bijections of conjugacy classes of $\sl_2$-subalgebras with nilpotent orbits. Proofs of Theorems \ref{sl2subalgebrabij} and \ref{psubalgebrathm}}\label{sec:psubs}

In this section we prove that there is a bijection
\[\{\text{conjugacy classes of }\sl_2\text{ subalgebras}\}\to\{\text{nilpotent orbits}\}\]
if and only if $p>b(G)$, where $b(G)$ is defined in the introduction. Here, the bijection is realised by sending a conjugacy class of $\sl_2$-subalgebras to the nilpotent orbit of largest dimension meeting it. It is not {\it a priori} clear that this would be well defined (as one conjugacy class of $\sl_2$ subalgebras could contain two non-conjugate nilpotent orbits of the same dimension) but we show that this never happens. 

We will need two lemmas.

\begin{lemma}\label{historal} If $p>b(G)$ then for any $\sl_2$-triple $(e,h,f)$ the elements $e$ and $f$ are nilpotent and the element $h$ is toral.\end{lemma}
\begin{proof}Suppose the statement is false. Then we may assume that $\h=\la e,h,f\ra\cong \sl_2$ is a non-$p$-subalgebra. As there are no Levi subalgebras of type $A_{rp-1}$, we have that $\h$ is non-$G$-cr by Lemma \ref{psubalgebraornongcr}. Thus $\h$ lives in a parabolic subalgebra $\p=\l+\q$ which may be chosen minimally subject to containing $\h$ such that the projection $\bar \h$ of $\h$ to a Levi subalgebra $\l$ is $G$-irreducible. By Lemma \ref{psubalgebraornongcr} again, it follows that $\bar\h$ is a $p$-subalgebra. In particular, the images $\bar e$ and $\bar f$ of $e$ and $f$ respectively are $p$-nilpotent. But as $e$ and $f$ are contained in the $p$-nilpotent spaces $\la\bar e\ra+\q$ and $\la\bar f\ra+\q$, respectively, they are also $p$-nilpotent. 

Now  $\h$ is a complement to $\q$ in the semidirect product $\bar\h+\q$ and $\q$ has a filtration by restricted $\l$-modules, thus a filtration by restricted $\bar\h$-modules. By (\ref{VinTable}), it follows that one of the $\bar\h$-composition factors of $\q$ is isomorphic to the $\bar\h$-module $L(p-2)$. We have $\opH^1(\sl_2,L(p-2))\cong k^2$. Let us describe a set of cocycle classes explicitly. Define $\gamma_{a,b}:\sl_2\to L(p-2)$ on a basis $e,h,f\in\sl_2$ via $\gamma_{a,b}(h)=0$, $\gamma_{a,b}(e)=av_{-p+2}$, $\gamma_{a,b}(f)=bv_{p-2}$ where $v_{p-2}$ and $v_{-p+2}$ are a chosen pair of highest and lowest weight vectors in $L(p-2)$. Then one checks that the $\gamma_{a,b}$ satisfy the cocycle condition (\ref{cocyclecondition}), so for instance
\[0=\gamma_{a,b}(h)=\gamma_{a,b}([e,f])=e\gamma_{a,b}(f)-f\gamma_{a,b}(e)=0-0.\] Further, if for some $v\in L(p-2)$, we have $\gamma_{a,b}(x)=\gamma_{a',b'}(x)+x(v)$ for all $x\in\sl_2$, then applying to $h$, we see that $h(v)=0$, so that $0$ is a weight of $L(p-2)$. This happens if and only if $p=2$, which is excluded from our analysis. Now, since $\opH^1(\sl_2,L(p-2))\cong k^2$, we see that the classes $[\gamma_{a,b}]$ are a basis for $\opH^1(\sl_2,L(p-2))$. In particular, in any equivalence class of cocycles in $\opH^1(\sl_2,L(p-2))$, there is a cocycle which vanishes on $h$. In the present situation, this means, following the argument in \ref{h1ofqis0}, (or simply by observing that the restriction map $\opH^1(\bar\h,\Lie(Q_i/Q_{i+1}))\to\opH^1(\pi(h),\Lie(Q_i/Q_{i+1}))$ is zero) that $\h$ can be replaced by a conjugate in which the element $h\in\h$ satisfies $\pi(h)=h$, so that in particular, $h$ is toral. 
\end{proof}

\begin{lemma}\label{uniquerep}Up to conjugacy by $G=\SO(V)$ or $\Sp(V)$, there is precisely one self-dual representation $V$ of $\h:=\la e,h,f\ra \cong \sl_2$ of dimension $r$ with $p<r<2p$. We may construct $V$ in such a way that $e\in\h$ acts with a single Jordan block of size $r$ and such that $f$ acts with Jordan blocks of sizes $(p-1-i),(i+1)^2$. Moreover, $V$ is uniserial with structure $L(i)|L(p-2-i)|L(i)$ where $r=p+i+1$. Conversely, if $V$ is uniserial with structure $L(i)|L(p-2-i)|L(i)$ then up to swapping $e$ and $f$, we may assume $e$ acts with a single Jordan block and $V$ is self-dual.\end{lemma}
\begin{proof}Concerning the first sentence of the lemma, the existence of a representation $V$ is implied by \cite{Pom80} applied to the regular nilpotent orbits in type $B_n$ and $C_n$ depending on the parity of $r$. Suppose $V$ did not consist of restricted composition factors. Then as $p<r<2p$, there would be precisely one which is not restricted, with at least one restricted factor. Thus $V$ would be decomposable and $e$ would certainly not act with a single Jordan block. Hence the composition factors of $V$ are restricted. 
If the socle were not simple, then there would be two linearly independent vectors $u,v$ for which $e\cdot u=e\cdot v=0$. But then the rank of $e$ cannot be $r-1$, contradicting the hypothesis that it acts with a single Jordan block.

Thus the socle is simple, isomorphic to $L(i)$ say. Then all the composition factors of $V$ are $L(i)$ or $L(p-2-i)$ since otherwise the vanishing of $\Ext^1$ between either of these and any other composition factor (by Lemma \ref{extA1}) would force a non-trivial direct summand. Also $V/\soc V$ must contain a submodule $L(p-2-i)$, or the socle would split off as a direct summand. Since $V$ contains the submodule $L(p-2-i)|L(i)$, self-duality forces it to contain a quotient $L(i)|L(p-2-i)$. As the simple socle is isomorphic to $L(i)$, the head is also simple and isomorphic to $L(i)$. Now, if there were two composition factors isomorphic to $L(p-2-i)$ then the dimension of $V$ would be at least $2p$, a contradiction. Thus the structure of $V$ is precisely $L(i)|L(p-2-i)|L(i)$, which has dimension $p+i+1$. Thus $r = p+i+1$, as required.

We must now show that there is just one such representation up to conjugacy. As $e$ has rank $r-1$, and $h$ is toral by \ref{historal}, there must be a vector of weight $-i$, say $w$, generating $V$ under $\h$. We have $\la e^{r-i-1}\cdot w,\dots, e^{r-1}\cdot w\ra$ spans $\soc V$, and if $r-p\leq j\leq r-1$ then $f\cdot e^j\cdot w\in\la e^{j-1}\cdot w\ra$, by uniqueness of the weights in the submodule $L(p-2-i)|L(i)$. Also by uniqueness of the weights of the quotient $L(p-2-i)|L(i)$ we must have $f\cdot e^i\cdot w \in\la e^{i-1}\cdot w, e^{i+p-1}\cdot w\ra$. As the endomorphism $e^p$ commutes with $e,f$ and $h$, the automorphism $v\mapsto v+e^p\cdot v$ of $V$ is an $\sl_2$-module homomorphism. Moreover, as $e$ preserves the form on $V$, so does $e^p$ and thus conjugating by an element $1+e^p$ of $G$ we may assume that $f\cdot e^j\cdot w \in\la e^{j-1}\cdot w\ra$ for some $0\leq j\leq i$. It is easy to check that this determines the action of $f$ completely, hence the action of $\h$.

Finally suppose $V$ is of the form $L(i)|L(p-2-i)|L(i)$. Consider the submodule $U$ with two composition factors. Then it is clear that either $e$ or $f$ has a Jordan block of size at least $p$ and by applying an automorphism if necessary, we may assume the former. Thus there is a vector $v_{-\bar\imath}$ of $h$-weight $-\bar\imath$ where $\bar\imath=p-2-i$ under which $e$ generates $U$. Let $v_{-\bar\imath+2s}:=e^s \cdot v_{-\bar\imath}$ for $0\leq s\leq p-1$ so that $U$ is spanned by the $v_j$. Inductively this determines the action of $f$ except for the action of $f$ on $v_{-\bar\imath}$ itself. Since $f$ sends this vector into the $-\bar\imath-2=i$th weight space of $U$ which is spanned by $v_i$, we have $f \cdot v_{-\bar\imath}=\lambda v_i$, for some $\lambda\in k$. Now take an $-i$th weight vector $w_{-i}\in V\setminus U$. Define $w_{-i+2s}=e^s \cdot w_{-i}$ for $0\leq s\leq i$. Then $V$ is spanned by the $w_j$ and $v_j$. Furthermore the action of $f$ is determined except on $w_{-i}$ itself. Since the $\bar\imath=p-2-i$ weight space is $1$-dimensional we must have $f \cdot w_{-i}=\mu v_{\bar \imath}$. But calculating $h \cdot w_{-i}=e \cdot f \cdot w_{-i} - f \cdot e \cdot w_{-i}$ we get $-iw_{-i}=\mu v_{-i} - iw_{-i}$. Thus $\mu=0$. Now the $w_j$ do not generate a submodule, so we must have $e \cdot w_i=\nu v_{-\bar\imath}$ for some $\nu\neq 0$. But then $h \cdot w_i = e \cdot f \cdot w_i - f \cdot e \cdot w_i$ gives us $iw_i=iw_i-\nu\lambda v_{-\bar\imath}$ so that $\lambda=0$. Replacing the $v_{j}$ by their division by $\nu$ we see that the action of $e,f$ and $h$ are now completely determined. It is easy to check directly that the module is self-dual and the Jordan blocks of $f$ are as claimed.\end{proof}

\begin{lemma}An indecomposable representation $V$ of $\h:=\la e,h,f\ra \cong \sl_2$ of dimension $r$ with $p<r<2p$ is either a quotient or submodule of a projective restricted module, or self-dual of the type described in Lemma \ref{uniquerep}.\label{2types}\end{lemma}
\begin{proof}A similar argument as used in the previous proof reduces us to the case that $V$ consists of restricted composition factors. We may assume $V$ contains at least one composition factor each of types $L(i),L(\bar \imath)=L(p-2-i)$. We will show that the multiplicity of $L(i)$ is $1$ or $2$. Assume $V$ has three composition factors isomorphic to $L(i)$. It suffices to show there is no representation $V$ with composition series $L(i)|L(\bar\imath)|(L(i)+L(i))$. Let $U$ be a module with composition series $L(\bar\imath)|(L(i)+L(i))$. Then it is easy to see that $e$ and $f$ satisfy $e^p=f^p=0$ as endomorphisms of $U$. Thus $U$ is a restricted representation. As it has simple head it is a quotient of the projective cover $P(\bar\imath)$ of $L(\bar\imath)$ with composition series $L(\bar\imath)|(L(i)+L(i))|L(\bar\imath)$. Let $v_{\bar\imath},\dots,v_{-\bar\imath}$ be a set of weight vectors for the $1$-dimensional weight spaces of $U$ coming from the composition factor $L(\bar \imath)$. Then one can check that an action of $\h$ on $U$ is given (up to isomorphism) by $e \cdot v_{\bar\imath}=u_{-i}$ and $f \cdot v_{-\bar\imath}=w_{i}$ where $\{u_{i},\dots,u_{-i}\}$ and $\{w_i,\dots,w_{-i}\}$ are a basis of weight vectors for the socle of $U$.

Now if $V$ exists of this form, we may take a weight vector $y_{-i}$ of weight $-i$ in $V\setminus U$. Then $y_{-i+2s}:=e^s \cdot y_i$ for $0\leq s\leq i-1$ together with $u,v$ and $w$ spans $V$. If $e \cdot y_i\neq 0$ then it is a non-zero multiple of $v_{\bar\imath}$. Now $h \cdot y_i=e \cdot f \cdot y_i - f \cdot e \cdot y_i$ leads to a contradiction. Similarly, we see that $f \cdot y_{-i}=0$. But this implies that $\la y_i,\dots,y_{-i}\ra$ is a submodule, which is a contradiction. 

Thus $V$ contains at most two composition factors isomorphic to $L(i)$. If it contains two, then it contains at most one of type $L(\bar\imath)$. It is now clear that there are only three possible structures for $V$, namely $L(\bar\imath)|(L(i)+L(i))$ or its dual, or $L(i)|L(\bar\imath)|L(i)$, which is unique up to isomorphism by Lemma \ref{uniquerep}.
\end{proof}

\begin{proof}[Proof of Theorem \ref{sl2subalgebrabij}]The reduction to the case where $G$ is simple is easy. If $G$ is classical, we may argue using the natural module $V$ for $G$. If $G$ is of type $A_n$, then whenever $p>b(A_n)=n+1$ we may appeal to Theorem \ref{jmthm}. On the other hand, whenever $p\leq b(A_n)$ there is an $A_{p-1}$-Levi subgroup $L$ of $G$, for which $\h$ acts on the natural module for $L'$ as an indecomposable module $k|L(p-2)$. Now there is a bijection between nilpotent orbits of $A_{p-1}$ and isomorphism classes of completely reducible $\sl_2$-representations via partitions of $p$, so that $\h$ is in an extra conjugacy class of $\sl_2$-subalgebras, showing that there is no bijection between the sets in (**). It remains to consider the cases where $G$ is of type $B_n,C_n$ or $D_n$. Note that $p>b(G)$ implies $2p>\dim V$ for $V$ the natural module for $G$. First note that completely reducible restricted actions of $\h$ are in $1$-$1$ correspondence with nilpotent orbits of $G$, which have associated partitions of $\dim V$ of size at most $p$, and these account for all completely reducible actions by Lemma \ref{psubalgebraornongcr}. Moreover, this bijection is realised by sending $\h$ to any of the ($G$-conjugate) nilpotent elements it contains.
If $V\downarrow\h$ is not completely reducible, we have that $V\downarrow\h$ contains an indecomposable summand $W$ which is not irreducible. Suppose $L(i)$ is a submodule of $W$. We have $0\leq i<p-1$ since $L(p-1)$ is in its own block and we can have at most one factor of this type in $V$ by dimensions. Then $U\cong L(p-2-i)|L(i)$ must be a submodule of $W$ also. If $U$ were a direct summand then $U^*\cong L(i)|L(p-2-i)$ is another submodule of $V$. But there can be no intersection between $U$ and $U^*$. This implies that $\dim V\geq 2p$, a contradiction. Hence $U$ is not a direct summand. 
By dimensions there is at most one indecomposable summand $W$, taking one of the forms discussed in Lemma \ref{2types}. If it is the quotient or submodule of a projective then a similar argument shows it has no intersection with its dual, contradicting the dimension of $W$. Hence $U$ lies in an indecomposable direct summand of the type discussed in Lemma \ref{uniquerep}. Being the unique such in $V$, it must be non-degenerate. Thus $\h$ lives in $\Lie(X\times Y)$ where $X$ is the stabiliser of $W^\perp$ in $G$, of type $\Sp(W^\perp)$, $\SO(W^\perp)$ or $\mathrm{O}(W^\perp)$ (as the case may be) with $Y$ a similar stabiliser of $W$ in $G$. Now, the necessarily restricted completely reducible image of $\h$ in $\Lie(X)$ is determined up to $X$-conjugacy by the image of any nilpotent element in $\h$. 
By Lemma \ref{uniquerep} the projection of $\h$ to $\Lie(Y)$ is determined uniquely up to conjugacy in $Y$, with an element of largest orbit dimension acting with a full Jordan block on $W$. In particular, a nilpotent element of $\h$ of largest orbit dimension always determines $\h$ up to conjugacy.

In case $G$ is exceptional, the analysis here is very case-by-case. Firstly, let $e$ be any nilpotent element of $\h$. By Lemma \ref{historal}, there is a toral element $h\in\h$ such that $[h,e]=2e$ and $h\in\im\ad e$. By Proposition \ref{findingcoch} we have $h\in\Lie(\tau(\Gm))$ for $\tau $ a cocharacter associated to $e$. Now in the grading of $\g$ associated with $\tau$, we have $e\in\g(2)$, $h\in\g(0)$ and since $[e,f]=h$, projecting $f$ to its component $\bar f\in\g(-2)$, we must have $(e,h,\bar f)$ an $\sl_2$-triple, with $f-\bar f\in\g_e=\g_e(\geq 0)$. As also $[h,f-\bar f]=-2(f-\bar f)$, we have $f-\bar f\in\bigoplus_{r>0}\g_{e}(-2+rp)$. If the subspace $\bigoplus_{r>0}\g_{e}(-2+rp)$ is trivial, we are automatically done. Looking at the tables in \cite{LT11}\footnote{In \cite{LT11}, the values of $r$ such that $\g_e(r)\neq 0$ are listed in the columns marked $m$.}, this already rules out $111$ of the $152$ orbits.

The strategy employed in the remaining cases is more subtle. The idea is to work  inductively through the remaining nilpotent orbits from largest dimension downwards, proving that for a given nilpotent element $e$ of orbit dimension $d$ there is just one conjugacy class of nilpotent elements $f$ whose orbit dimension is $d$ or lower and such that $(e,h,f)$ is an $\sl_2$-triple. That is, we will show that whenever $f$ is not conjugate to $\bar f$ by an element in $G_e\cap G_h$, then $f$ has higher orbit dimension than that of $e$. To show this, we will effectively find all possible $f$ such that $(e,h,f)$ is an $\sl_2$-triple and check each case.

To progress further, recall that $G_e$ is the semidirect product $C_eR_e$ of its reductive part $C_e$ and unipotent radical $R_e$. Since $p$ is a good prime at this stage, one has $\Lie(C_e)=\g_e(0)$ and $\Lie(R_e)=\g_e(>0)$. We also have $\Lie(G_e\cap G_h)\subseteq \bigoplus_{r\geq 0}\g_e(rp)$. We present two tools, which together deal with the remaining cases. For the first, henceforth Tool (a), suppose $C_e$ acts with finitely many orbits on the subspace $\bigoplus_{r>0}\g_e(-2+rp)$. In these cases one can write down all possibilities for the element $f$ up to conjugacy by $G_e\cap G_h$. Then it is a simple matter to check the Jordan blocks of the element $f$ on the adjoint module for $\g$ in GAP and observe, by comparing with \cite{Law95}, that the orbit dimension is larger than that of $e$. For example, if $(\g, \OO, p)$ is $(E_7, (A_5)', 11)$ then we may take \def\arraystretch{0.5} \arraycolsep=0pt
\begin{align*}e&=e_{\Small\begin{array}{c c c c c c c}1&0&0&0&0&0\\&&0\end{array}}+
e_{\Small\begin{array}{c c c c c c c}0&1&0&0&0&0\\&&0\end{array}}+
e_{\Small\begin{array}{c c c c c c c}0&0&1&0&0&0\\&&0\end{array}}+
e_{\Small\begin{array}{c c c c c c c}0&0&0&1&0&0\\&&0\end{array}}+
e_{\Small\begin{array}{c c c c c c c}0&0&0&0&1&0\\&&0\end{array}},\\
h&=5\cdot h_{\alpha_1}+
8\cdot h_{\alpha_3}+
9\cdot h_{\alpha_4}+
8\cdot h_{\alpha_5}+
5\cdot h_{\alpha_6},\\
\bar f&=5\cdot e_{\Small\begin{array}{c c c c c c c}-1&0&0&0&0&0\\&&0\end{array}}+
8\cdot e_{\Small\begin{array}{c c c c c c c}-0&1&0&0&0&0\\&&0\end{array}}+
9\cdot e_{\Small\begin{array}{c c c c c c c}-0&0&1&0&0&0\\&&0\end{array}}+
8\cdot e_{\Small\begin{array}{c c c c c c c}-0&0&0&1&0&0\\&&0\end{array}}+
5\cdot e_{\Small\begin{array}{c c c c c c c}-0&0&0&0&1&0\\&&0\end{array}}.\end{align*}
We have $C:=C_e^\circ$ of type $A_1^2$ and $\g_e(-2+p)$ is a module for $C$ of high weight $\varpi_1$ for the first factor say. Thus $C$ has two orbits on $\g_e(-2+p)$, namely the zero orbit and the non-zero orbit. The element $\bar f$ itself corresponds to the zero orbit, whereas if $f=\bar f+f_1$ for $0\neq f_1\in\g_e(-2+p)$ then one checks that the Jordan blocks of the action of $f$ on $\g$ are $23+17^3+15+11+9^3+3+1^3$, whereas those of $\bar f$ are $11+10^2+9^3+7+6^6+5^3+4^2+3+1^6$. Comparing with \cite{Law95} one sees that $f$ is in the orbit $E_6$ whereas $\bar f$ is in the orbit $(A_5)'$.

The remaining cases all have the property that $\g_e(p)\neq 0$ with $\g_e(p)$ having a basis of commuting sums of root vectors. To describe Tool (b), suppose $x\in\g_e(p)$ is a sum of commuting root vectors. Then as $p\neq 2$, one may form the endomorphism $\delta_x:=1+\ad x+\frac{1}{2}(\ad x)^2$. Since $x\in \g_e(p)$, it follows that $x$ commutes with both $e$ and $h$, hence $\delta_x\in G_e\cap G_h$.  Thus $\la e,h,\delta_x(\bar f)\ra$ is an $\sl_2$-triple. For any $y\in \g_e(p)$, we have $[\delta_x(\bar f),\delta_y(\bar f)]=0$ modulo $\g_e(>-2+p)$ and so we get a linear map $\delta_\bullet(\bar f):\g_e(p)\to\g_e(-2+p)$ by $x\mapsto \delta_x(\bar f)$. Now if $z\in\g_e(0)$ then as $\bar f\in\g_e(-2)$, we have $[\bar f,z] \in \g(-2)$ and since $[e,[\bar f,z]]=[h,z]=0$, we have $[\bar f,z]\in\g_e(-2)=0$. Thus $\bar f$ is in the centraliser of $\g_e(0)$ so that $C_e^\circ$ also commutes with $\delta_\bullet(\bar f)$; this means that $\delta_\bullet(\bar f)$ is a $C_e^\circ$-module map from $\g_e(p)\to\g_e(-2+p)$. Thus one may assume, replacing $f$ by a conjugate, that the projection of $f$ to the image of $\delta_\bullet(f)$ in $\g_e(-p+2)$ is zero. In particular, if this map is an isomorphism, one concludes that any $\sl_2$-triple $(e,h,f)$ is conjugate to another $(e,h,f')$ such that the projection of $f'$ to $\g_e(-2+p)$ is zero. If $\bigoplus_{r>0}\g_e(-2+rp)=\g_e(-2+p)$ (which it almost always is) this shows that $f$ is unique up to conjugacy. When using the fact that $\delta_\bullet(\bar f)$ is a $C_e^\circ$-module map, to check the isomorphism, one finds that it always suffices to check that $\delta_\bullet(\bar f)$ is non-zero on restriction to high weights of the $C_e^\circ$-modules $\g_e(p)$, since these modules are always semisimple. 

For example, suppose $(\g,\OO,p)$ is $(E_6,D_5(a_1),7)$. Checking \cite{LT11}, one may choose \[e=e_{\Small\begin{array}{c c c c c c c}1&0&0&0&0\\&&0\end{array}}+
e_{\Small\begin{array}{c c c c c c c}0&0&0&0&0\\&&1\end{array}}+
e_{\Small\begin{array}{c c c c c c c}0&1&0&0&0\\&&0\end{array}}+
e_{\Small\begin{array}{c c c c c c c}0&0&0&1&0\\&&0\end{array}}+
e_{\Small\begin{array}{c c c c c c c}0&0&1&0&0\\&&1\end{array}}+
e_{\Small\begin{array}{c c c c c c c}0&0&1&1&0\\&&0\end{array}}\] 
and
\[h=6\cdot h_{\alpha_1}+
7\cdot h_{\alpha_2}+
10\cdot h_{\alpha_3}+
12\cdot h_{\alpha_4}+
7\cdot h_{\alpha_5}.\] Under these circumstances, one may calculate that the component of $f$ in $\g(-2)$ is 
\[\bar f:=6\cdot e_{\Small\begin{array}{c c c c c c c}-1&0&0&0&0\\&&0\end{array}}+
e_{\Small\begin{array}{c c c c c c c}-0&0&0&0&0\\&&1\end{array}}+
10\cdot e_{\Small\begin{array}{c c c c c c c}-0&1&0&0&0\\&&0\end{array}}+
e_{\Small\begin{array}{c c c c c c c}-0&0&0&1&0\\&&0\end{array}}+
6\cdot e_{\Small\begin{array}{c c c c c c c}-0&0&1&0&0\\&&1\end{array}}+
6\cdot e_{\Small\begin{array}{c c c c c c c}-0&0&1&1&0\\&&0\end{array}}.\] 
From \cite{LT11}, one has $\bigoplus_{r>0}\g_e(-2+rp)=\g_e(5)$ of dimension $2$, while $\g_e(p)$ is generated over $k$ by the (commuting) root vectors $x_1:=e_{\Small\begin{array}{ccccc}-0&0&0&0&1\\&&0&&\end{array}}$ and $x_2:=e_{\Small\begin{array}{ccccc}1&2&3&2&1\\&&2&&\end{array}}$. One checks that the images of $x_1$ and $x_2$ under $\delta_\bullet(f)$ are linearly independent. Thus $\delta_\bullet(f)$ induces an isomorphism $\g_e(p)\to\g_e(-2+p)$ as required, showing the uniqueness of $f$ in this case.

Lastly, using Tool (b) to assume that $f$ projects trivially to the image of $\delta_\bullet(f)$ and then applying Tool (a) finishes the analysis in any remaining cases where $\g_e(-2+p)\neq 0$. For example, if $(\g,\OO,p)=(E_8,E_7(a_4),11)$, then $C_e^\circ$ is of type $A_1$ and we may take \[e=e_{\Small\begin{array}{c c c c c c c}1&0&0&0&0&0&0\\&&0\end{array}}+
e_{\Small\begin{array}{c c c c c c c}0&0&1&0&0&0&0\\&&0\end{array}}+
e_{\Small\begin{array}{c c c c c c c}0&0&0&0&0&1&0\\&&0\end{array}}+
e_{\Small\begin{array}{c c c c c c c}0&0&0&0&1&1&0\\&&0\end{array}}+
e_{\Small\begin{array}{c c c c c c c}0&1&1&0&0&0&0\\&&1\end{array}}+
e_{\Small\begin{array}{c c c c c c c}0&0&1&1&0&0&0\\&&1\end{array}}+
e_{\Small\begin{array}{c c c c c c c}0&1&1&1&0&0&0\\&&0\end{array}}+
e_{\Small\begin{array}{c c c c c c c}0&0&1&1&1&0&0\\&&1\end{array}}.\] 
Then $\delta_\bullet(f)$ acts non-trivially on the element $x=e_{\Small\begin{array}{ccccccc}2&4&6&5&4&3&1\\&&3&&\end{array}}\in\g_e(p)$ inducing a $C_e^\circ$-isomorphism $\g_e(p)\to kC_e^\circ\cdot(e_{\Small\begin{array}{ccccccc}2&4&6&5&4&2&1\\&&3&&\end{array}}+e_{\Small\begin{array}{ccccccc}2&4&6&5&3&2&1\\&&3&&\end{array}})\subseteq\g_e(-2+p)$ and so one may assume \[f\in\bar f+kC_e^\circ\cdot (e_{\Small\begin{array}{ccccccc}2&4&6&5&4&3&1\\&&3&&\end{array}}),\] with $kC_e^\circ\cdot (e_{\Small\begin{array}{ccccccc}2&4&6&5&4&3&1\\&&3&&\end{array}})\cong L(1)$ as a $C_e^\circ$-module. Now, we use Tool (a). As $C_e^\circ$ has just one non-zero orbit on the representation $L(1)$, we may assume that $f=\bar f+e_{\Small\begin{array}{ccccccc}2&4&6&5&4&3&1\\&&3&&\end{array}}$. But computing the Jordan blocks of $f$ on the adjoint representation, one finds that $f$ has a higher orbit dimension than that of $\bar f$.

There remain some cases where $\g_e(-2+2p)\neq 0$. These are $(E_8,E_8(a_5),11)$, $(E_8,E_8(b_4),11)$ and $(F_4,F_4(a_2),5)$. Since these cases are distinguished orbits, $\g_e(-2+p)=0$. Precisely the same analysis as used in Tool (b) will work here, replacing $(\g_e(-2+p),\g_e(p))$ with $(\g_e(-2+2p),\g_e(2p))$.

It remains to show that no bijection exists when $2 < p \leq b(G)$. It is well known that the number of nilpotent orbits of $\g$ is finite and so it suffices to show there are infinitely many classes of $\sl_2$-subalgebras. Suppose $\l$ is a Lie subalgebra of type $A_{p-1}$ and let $e = \sum_{i=1}^{p-1} e_{\alpha_i}$ and $f = \sum_{i=0}^{p-1} -i^2 e_{-\alpha_i}$. Then one checks that $\la e, f \ra$ is an $\sl_2$-subalgebra with $[e,f] = \text{diag}({p-1}, {p-3}, \dots,$ $-p+3, -p+1)$. Further, let $f_0 = \sum_{i=0}^{p-1} i e_{-\alpha_i}$ and $\lambda \in k$ with $\lambda^p \neq \lambda$. Then again one checks that $\la e, (f+ \lambda f_0) \ra$ is an $\sl_2$-subalgebra, this time with $[e,f] = h + \lambda I$. We therefore have infinitely many $\sl_2$-subalgebras with pairwise non-isomorphic representations on the restriction of the natural representation of $\l$.    
The condition $2 < p \leq b(G)$ implies that if $G$ is not of type $G_2$ then $\g$ has a Levi subalgebra of type $A_{p-1}$ and when $G$ is of type $G_2$ then $p=3$ and $\g$ has a pseudo-Levi subalgebra of type $A_2$. In all cases we therefore have a subalgebra $\l$ of type $A_{p-1}$ and moreover, the restriction of the adjoint representation of $\g$ to $\l$ contains a copy of the natural representation of $\l$. Thus we have infinitely many $GL(\g)$-conjugacy classes of $\sl_2$-subalgebras of $\g$ (all with pairwise non-isomorphic representations) and so we certainly have infinitely many $G$-conjugacy classes of $\sl_2$-subalgebras.  \end{proof}

With Theorem \ref{sl2subalgebrabij} in hand, we turn our attention to $p$-subalgebras and prove the last result, Theorem \ref{psubalgebrathm}. 

\begin{proof}[Proof of Theorem \ref{psubalgebrathm}]
In light of Theorem \ref{maintheorem}, it suffices to prove the following three claims: All $p$-subalgebras of type $A_1$ are $G$-cr when $b(G) < p \leq h(G)$; there exists a non-$G$-cr $p$-subalgebra of type $A_1$ when $5 \leq p \leq b(G)$; and the examples of non-$G$-cr subalgebras of type $B_2$, $G_2$ and $W_1$ given in Section \ref{sec:proofofmain} are all $p$-subalgebras. 

First, let $b(G) < p \leq h(G)$ and $\h = \la e,h,f \ra$ be a $p$-subalgebra of $\g$, with $e$ belonging to the largest nilpotent orbit meeting $\h$. Since $\h$ is a $p$-subalgebra we have $e^{[p]}=0$ and the restriction on $p$ implies that $p$ is a good prime for $G$. As in the proof of Theorem \ref{gooda1s}, we are therefore furnished with a good $A_1$-subgroup $H$ of $G$ such that $e \in \Lie(H)$, by \cite[Props.~4.1 \& 4.2]{Sei00}. But now Theorem \ref{maintheorem} implies that $\Lie(H)$ is conjugate to $\h$, since both contain $e$ and all nilpotent elements of $\Lie(H)$ are conjugate (since all unipotent elements of $H$ are conjugate) so $e$ belongs to the largest dimensional nilpotent orbit meeting $\Lie(H)$. By \cite[Prop.~7.2]{Sei00}, good $A_1$-subgroups are $G$-cr. Therefore $\Lie(H)$ is $G$-cr by \cite[Thm.~1]{McN07} and hence $\h$ is $G$-cr, as required.  

Now suppose $5 \leq p \leq b(G)$. This implies that either $G$ is of type $E_6$ and $p=5$ or $G$ is of type $E_7, E_8$ and $p=5,7$. In each case we present a non-$G$-cr $p$-subalgebra of type $A_1$: By definition of $b(G)$, we have a Levi subgroup of type $A_{p-1}$. We know $\l = \Lie(L)$ has a $p$-subalgebra of type $\sl_2$ acting as $L(p-2) | k$ on the natural module for $\l$ (see the proof of Lemma \ref{mtfora1a}). This subalgebra is therefore non-$L$-cr and hence non-$G$-cr by Lemma \ref{levinongcr}.

Finally, we consider the subalgebras of type $B_2$, $G_2$ and $W_1$ from Section \ref{sec:proofofmain}. The claim is clear for the subalgebras of type $B_2$ and $G_2$, since the given examples are $\Lie(H)$ for $H$ a subgroup of $G$, hence $p$-subalgebras. So it remains to consider the subalgebras of type $W_1$ constructed in Section \ref{sec:w1s}. Firstly, the subalgebra of type $W_1$ contained in $A_4$ when $p=5$ is a $p$-subalgebra since it acts via its canonical representation $k[X]/X^p \cong L(4)|k$. For $p=7$ we have two explicit constructions of subalgebras of type $W_1$, one contained in a $C_3$-parabolic of $F_4$ and the other contained in an $A_5$-parabolic of $E_6$ (and $E_7$, $E_8$). One checks in GAP that the elements representing $\del$ and $X^{p-1} \del$ are both sent to $0$ by the $[p]$-map in both cases and hence the subalgebras are $p$-subalgebras. 
\end{proof}

\begin{remark}
Suppose $G$ is simple and of exceptional type with Lie algebra $\g$. Then we can extend Theorem \ref{psubalgebrathm} to the case $p=3$ when $\h$ is of type $A_1$. We have checked computationally that all $p$-subalgebras of type $A_1$ are $G$-cr when $G$ is of type $G_2$. If $G$ is not of type $G_2$ then it has a Levi subgroup $L$ of type $A_{2}$. We know $\l = \Lie(L)$ has a $p$-subalgebra of type $\sl_2$ acting as $L(1) | k$ on the natural $3$-dimensional module. This subalgebra is therefore non-$L$-cr and hence non-$G$-cr by Lemma \ref{levinongcr}. 
\end{remark}

\begin{remark}The proofs and statements of the main theorems  show that whenever $p>b(G)$, there is in fact a bijection of the form (*) between conjugacy classes of $p$-subalgebras of type $\sl_2$ and nilpotent elements $e\in\g=\Lie(G)$ such that $e^{[p]}=0$. \end{remark}

Finally, we turn our attention to the proof of Theorem \ref{jm}. 

\begin{proof}[Proof of Theorem \ref{jm}]
If the nilpotent element $e$ belongs to an orbit whose label is defined over $\C$ then we will see that $e$ is contained in a $\Z$-defined $\sl_2$-subalgebra of $\g_\C$ such that the image of $e$ in $\g_\Z\otimes_\Z\F_p=\g_{\F_p}$ is of the same type as $e$. In particular $e$ is always contained in an $\sl_2$-subalgebra of $\g$. If $G$ is of classical type, this is straightforward as the natural representation is defined over $\Z$ such that $e$ is a sum of Jordan blocks with $1$ on the super-diagonal, $h$ is diagonal and $f$ is determined by a combinatorial formula given in \cite[Proof of Prop.~5.3.1]{Car93}. More carefully, for sufficiently large $p$, there is a $1$-$1$ correspondence between partitions of $n$ of an appropriate sort, depending on the root system of $G$, and direct sums of irreducible $\sl_2$-representations of dimensions corresponding to the partition which give an embedding into $G$. The irreducible representations satisfy the condition of being defined over $\Z$, with $e$ represented by a sum of Jordan blocks with $1$s on the superdiagonal. Thus we may assume that $G$ is exceptional. Now \cite[Lem.~2.1,~2.4]{Tes95} gives a $\Z$-defined $\sl_2$-triple containing $e$ in the cases that $e$ is not contained in a maximal rank subalgebra. (That the representatives for the distinguished nilpotent elements in \cite{Tes95} have the same label over all primes follows from the results of \cite{LS12}.) Of course, if $e$ is in a maximal rank subalgebra then we are done by induction. 


We are left with the distinguished elements that belong to orbits whose labels are not defined over $\C$. There are just two of these for $p\geq 3$: one is the exceptional nilpotent orbit in $G_2$ with label $A_1^{(3)}$. The other is the exceptional nilpotent orbit in $E_8$ with label $A_7^{(3)}$. For the latter we simply exhibit an example cooked up with GAP. From \cite{LS12} one may take\def\arraystretch{0.5} \arraycolsep=0pt
\[e:=e_{\Small\begin{array}{c c c c c c c c}0&0&0&1&1&1&0\\&&0\end{array}}+
e_{\Small\begin{array}{c c c c c c c c}0&0&0&0&1&1&1\\&&0\end{array}}+
e_{\Small\begin{array}{c c c c c c c c}1&1&1&0&0&0&0\\&&1\end{array}}+
e_{\Small\begin{array}{c c c c c c c c}1&1&1&1&0&0&0\\&&0\end{array}}+
e_{\Small\begin{array}{c c c c c c c c}0&0&1&1&1&0&0\\&&1\end{array}}+
e_{\Small\begin{array}{c c c c c c c c}0&1&1&1&1&0&0\\&&0\end{array}}+
e_{\Small\begin{array}{c c c c c c c c}0&1&2&1&0&0&0\\&&1\end{array}}+
e_{\Small\begin{array}{c c c c c c c c}0&0&1&1&1&1&1\\&&0\end{array}}\]
and this is filled out to an $\sl_2$-triple with
\begin{align*} h&=h_{\alpha_4}+
h_{\alpha_5}+
2\cdot h_{\alpha_6}+
h_{\alpha_8}\\
f&=2\cdot e_{\Small\begin{array}{c c c c c c c c}-0&0&0&1&1&1&0\\&&0\end{array}}\
+
e_{\Small\begin{array}{c c c c c c c c}-0&0&1&1&1&0&0\\&&1\end{array}}+
e_{\Small\begin{array}{c c c c c c c c}-0&1&1&1&1&0&0\\&&0\end{array}}+
e_{\Small\begin{array}{c c c c c c c c}-1&1&1&1&1&0&0\\&&0\end{array}}+
2\cdot e_{\Small\begin{array}{c c c c c c c c}-0&1&2&1&0&0&0\\&&1\end{array}}\\&+
2\cdot e_{\Small\begin{array}{c c c c c c c c}-0&0&1&1&1&1&0\\&&1\end{array}}+
e_{\Small\begin{array}{c c c c c c c c}-0&1&1&1&1&1&0\\&&0\end{array}}+
e_{\Small\begin{array}{c c c c c c c c}-0&0&1&1&1&1&1\\&&0\end{array}}+
2\cdot e_{\Small\begin{array}{c c c c c c c c}-1&1&2&1&0&0&0\\&&1\end{array}}+
2\cdot e_{\Small\begin{array}{c c c c c c c c}-0&1&2&1&1&0&0\\&&1\end{array}}.\end{align*}

Finally, if $e$ is of type $A_1^{(3)}$ in $G_2$ then $e$ can be taken to be $e_{2\alpha_1+\alpha_2}+e_{3\alpha_{1}+2\alpha_2}$. Now it is straightforward to check that the image of $(\ad e)^2$ does not contain $e$. (One can even do this by hand.)
\end{proof}

{\footnotesize
\bibliographystyle{amsalpha}
\bibliography{bib}}

\providecommand{\bysame}{\leavevmode\hbox to3em{\hrulefill}\thinspace}
\providecommand{\MR}{\relax\ifhmode\unskip\space\fi MR }
\providecommand{\MRhref}[2]{%
  \href{http://www.ams.org/mathscinet-getitem?mr=#1}{#2}
}
\providecommand{\href}[2]{#2}
\begin{thebibliography}{BMRT13}

\bibitem[ABS90]{ABS90}
H.~Azad, M.~Barry, and G.~Seitz, \emph{On the structure of parabolic
  subgroups}, Comm. Algebra \textbf{18} (1990), no.~2, 551--562. \MR{MR1047327
  (91d:20048)}

\bibitem[BMRT13]{BMRT13}
M.~Bate, B.~Martin, G.~R{{\"o}}hrle, and R.~Tange, \emph{Closed orbits and
  uniform {$S$}-instability in geometric invariant theory}, Trans. Amer. Math.
  Soc. \textbf{365} (2013), no.~7, 3643--3673. \MR{3042598}

\bibitem[BNP02]{BNP02}
C.~P. Bendel, D.~K. Nakano, and C.~Pillen, \emph{Extensions for finite
  {C}hevalley groups. {II}}, Trans. Amer. Math. Soc. \textbf{354} (2002),
  no.~11, 4421--4454 (electronic). \MR{1926882 (2003k:20063)}

\bibitem[BNP04]{BNP04-Frob}
\bysame, \emph{Extensions for {F}robenius kernels}, J. Algebra \textbf{272}
  (2004), no.~2, 476--511. \MR{2028069 (2004m:20089)}

\bibitem[BNW09]{BNW09}
B.~D. Boe, D.~K. Nakano, and E.~Wiesner, \emph{{$\rm Ext^1$}-quivers for the
  {W}itt algebra {$W(1,1)$}}, J. Algebra \textbf{322} (2009), no.~5,
  1548--1564. \MR{2543622 (2011b:17027)}

\bibitem[Car72]{Car72}
R.~W. Carter, \emph{Conjugacy classes in the {W}eyl group}, Compositio Math.
  \textbf{25} (1972), 1--59. \MR{0318337 (47 \#6884)}

\bibitem[Car93]{Car93}
Roger~W. Carter, \emph{Finite groups of {L}ie type}, Wiley Classics Library,
  John Wiley \& Sons Ltd., Chichester, 1993, Conjugacy classes and complex
  characters, Reprint of the 1985 original, A Wiley-Interscience Publication.
  \MR{MR1266626 (94k:20020)}

\bibitem[DeB02]{DeB02}
Stephen DeBacker, \emph{Parametrizing nilpotent orbits via {B}ruhat-{T}its
  theory}, Ann. of Math. (2) \textbf{156} (2002), no.~1, 295--332. \MR{1935848}

\bibitem[Far09]{Far09}
R.~Farnsteiner, \emph{Group-graded algebras, extensions of infinitesimal
  groups, and applications}, Transform. Groups \textbf{14} (2009), no.~1,
  127--162. \MR{2480855}

\bibitem[HS16a]{HSMax}
Sebastian Herpel and David~I. Stewart, \emph{Maximal subalgebras of {C}artan
  type in the exceptional {L}ie algebras}, Selecta Math. (N.S.) \textbf{22}
  (2016), no.~2, 765--799. \MR{3477335}

\bibitem[HS16b]{HS14}
Sebastian Herpel and David~I. Stewart, \emph{On the smoothness of normalisers,
  the subalgebra structure of modular {L}ie algebras and the cohomology of
  small representations}, Doc. Math. (2016), to appear.

\bibitem[Jan03]{Jan03}
J.~C. Jantzen, \emph{Representations of algebraic groups}, second ed.,
  Mathematical Surveys and Monographs, vol. 107, American Mathematical Society,
  Providence, RI, 2003. \MR{MR2015057 (2004h:20061)}

\bibitem[Jan04]{Jan04}
\bysame, \emph{Nilpotent orbits in representation theory}, Lie theory, Progr.
  Math., vol. 228, Birkh{\"a}user Boston, Boston, MA, 2004, pp.~1--211.
  \MR{2042689 (2005c:14055)}

\bibitem[Kos59]{Kos59}
Bertram Kostant, \emph{The principal three-dimensional subgroup and the {B}etti
  numbers of a complex simple {L}ie group}, Amer. J. Math. \textbf{81} (1959),
  973--1032. \MR{0114875 (22 \#5693)}

\bibitem[Law95]{Law95}
R.~Lawther, \emph{Jordan block sizes of unipotent elements in exceptional
  algebraic groups}, Comm. Algebra \textbf{23} (1995), no.~11, 4125--4156.
  \MR{MR1351124 (96h:20084)}

\bibitem[LS96]{LS96}
M.~W. Liebeck and G.~M. Seitz, \emph{Reductive subgroups of exceptional
  algebraic groups}, Mem. Amer. Math. Soc. \textbf{121} (1996), no.~580,
  vi+111. \MR{MR1329942 (96i:20059)}

\bibitem[LS04]{LS04}
\bysame, \emph{The maximal subgroups of positive dimension in exceptional
  algebraic groups}, Mem. Amer. Math. Soc. \textbf{169} (2004), no.~802,
  vi+227. \MR{MR2044850 (2005b:20082)}

\bibitem[LS12]{LS12}
Martin~W. Liebeck and Gary~M. Seitz, \emph{Unipotent and nilpotent classes in
  simple algebraic groups and {L}ie algebras}, Mathematical Surveys and
  Monographs, vol. 180, American Mathematical Society, Providence, RI, 2012.
  \MR{2883501}

\bibitem[LT04]{LT04}
Martin~W. Liebeck and Donna~M. Testerman, \emph{Irreducible subgroups of
  algebraic groups}, Q. J. Math. \textbf{55} (2004), no.~1, 47--55. \MR{2043006
  (2005b:20087)}

\bibitem[LT11]{LT11}
R.~Lawther and D.~M. Testerman, \emph{Centres of centralizers of unipotent
  elements in simple algebraic groups}, Mem. Amer. Math. Soc. \textbf{210}
  (2011), no.~988, vi+188. \MR{2780340 (2012c:20127)}

\bibitem[L{\"u}b01]{Lub01}
Frank L{\"u}beck, \emph{Small degree representations of finite {C}hevalley
  groups in defining characteristic}, LMS J. Comput. Math. \textbf{4} (2001),
  135--169 (electronic). \MR{1901354 (2003e:20013)}

\bibitem[LY93]{LY93}
Jia~Chun Liu and Jia~Chen Ye, \emph{Extensions of simple modules for the
  algebraic group of type {$G_2$}}, Comm. Algebra \textbf{21} (1993), no.~6,
  1909--1946. \MR{1215553 (94h:20051)}

\bibitem[McN05]{McN05}
G.~J. McNinch, \emph{Optimal {${\rm SL}(2)$}-homomorphisms}, Comment. Math.
  Helv. \textbf{80} (2005), no.~2, 391--426. \MR{2142248 (2006f:20055)}

\bibitem[McN07]{McN07}
George McNinch, \emph{Completely reducible {L}ie subalgebras}, Transformation
  Groups \textbf{12} (2007), no.~1, 127--135.

\bibitem[MT11]{MT}
Gunter Malle and Donna Testerman, \emph{Linear algebraic groups and finite
  groups of {L}ie type}, Cambridge Studies in Advanced Mathematics, vol. 133,
  Cambridge University Press, Cambridge, 2011. \MR{2850737 (2012i:20058)}

\bibitem[Pom80]{Pom80}
Klaus Pommerening, \emph{\"{U}ber die unipotenten {K}lassen reduktiver
  {G}ruppen. {II}}, J. Algebra \textbf{65} (1980), no.~2, 373--398. \MR{585729
  (83d:20031)}

\bibitem[Pre91]{Pre91}
A.~A. Premet, \emph{The {G}reen ring of a simple three-dimensional {L}ie
  {$p$}-algebra}, Izv. Vyssh. Uchebn. Zaved. Mat. (1991), no.~10, 56--67.
  \MR{1179217}

\bibitem[Pre95a]{Pre95}
Alexander Premet, \emph{An analogue of the {J}acobson-{M}orozov theorem for
  {L}ie algebras of reductive groups of good characteristics}, Trans. Amer.
  Math. Soc. \textbf{347} (1995), no.~8, 2961--2988. \MR{1290730 (95k:17012)}

\bibitem[Pre95b]{Pre95KW}
\bysame, \emph{Irreducible representations of {L}ie algebras of reductive
  groups and the {K}ac-{W}eisfeiler conjecture}, Invent. Math. \textbf{121}
  (1995), no.~1, 79--117. \MR{1345285 (96g:17007)}

\bibitem[Pre03]{Pre03}
\bysame, \emph{Nilpotent orbits in good characteristic and the
  {K}empf-{R}ousseau theory}, J. Algebra \textbf{260} (2003), no.~1, 338--366,
  Special issue celebrating the 80th birthday of Robert Steinberg. \MR{1976699}

\bibitem[PS]{PSMax}
Alexander Premet and David~I. Stewart, \emph{Classification of the maximal
  subalgebras of exceptional lie algebras over fields of good characteristic},
  in preparation.

\bibitem[PS16]{PreSteNilp}
\bysame, \emph{Rigid orbits and sheets in reductive lie algebras over fields of
  prime characteristic}, J. Inst. Math. Jussieu (2016), to appear.

\bibitem[Sei00]{Sei00}
Gary~M. Seitz, \emph{Unipotent elements, tilting modules, and saturation},
  Invent. Math. \textbf{141} (2000), no.~3, 467--502. \MR{1779618
  (2001j:20074)}

\bibitem[Ser05]{Ser05}
J-P. Serre, \emph{Compl{\`e}te r{\'e}ductibilit{\'e}}, Ast{\'e}risque (2005),
  no.~299, Exp. No. 932, viii, 195--217, S{{\'e}}minaire Bourbaki. Vol.
  2003/2004. \MR{2167207 (2006d:20084)}

\bibitem[SF88]{SF88}
H.~Strade and R.~Farnsteiner, \emph{Modular {L}ie algebras and their
  representations}, Monographs and Textbooks in Pure and Applied Mathematics,
  vol. 116, Marcel Dekker Inc., New York, 1988. \MR{929682 (89h:17021)}

\bibitem[SS70]{SpSt70}
T.~A. Springer and R.~Steinberg, \emph{Conjugacy classes}, Seminar on
  {A}lgebraic {G}roups and {R}elated {F}inite {G}roups ({T}he {I}nstitute for
  {A}dvanced {S}tudy, {P}rinceton, {N}.{J}., 1968/69), Lecture Notes in
  Mathematics, Vol. 131, Springer, Berlin, 1970, pp.~167--266. \MR{0268192 (42
  \#3091)}

\bibitem[Ste10]{SteSL2}
David~I. Stewart, \emph{The second cohomology of simple {${\rm
  SL}_2$}-modules}, Proc. Amer. Math. Soc. \textbf{138} (2010), no.~2,
  427--434. \MR{2557160 (2011b:20134)}

\bibitem[Ste12]{SteSL3}
D.~I. Stewart, \emph{The second cohomology of simple {$SL_3$}-modules}, Comm.
  Algebra \textbf{40} (2012), no.~12, 4702--4716. \MR{2989676}

\bibitem[Ste16]{SteComp}
David~I. Stewart, \emph{On the minimal modules for exceptional {L}ie algebras:
  {J}ordan blocks and stabilizers}, LMS Journal of Computation and Mathematics
  \textbf{19} (2016), no.~1, 235--258.

\bibitem[Str73]{Str73}
H.~Strade, \emph{Lie algebra representations of dimension {$p-1$}}, Proc. Amer.
  Math. Soc. \textbf{41} (1973), 419--424. \MR{0330247 (48 \#8585)}

\bibitem[Tes92]{Tes92}
D.~M. Testerman, \emph{The construction of the maximal {$A_1$}'s in the
  exceptional algebraic groups}, Proc. Amer. Math. Soc. \textbf{116} (1992),
  no.~3, 635--644. \MR{1100666 (93a:20073)}

\bibitem[Tes95]{Tes95}
Donna~M. Testerman, \emph{{$A_1$}-type overgroups of elements of order {$p$} in
  semisimple algebraic groups and the associated finite groups}, J. Algebra
  \textbf{177} (1995), no.~1, 34--76. \MR{1356359 (96j:20067)}

\bibitem[Wei94]{Wei94}
Charles~A. Weibel, \emph{An introduction to homological algebra}, Cambridge
  Studies in Advanced Mathematics, vol.~38, Cambridge University Press,
  Cambridge, 1994. \MR{1269324 (95f:18001)}

\bibitem[Ye90]{Ye90}
Jia~Chen Ye, \emph{Extensions of simple modules for the group {${\rm
  Sp}(4,K)$}}, J. London Math. Soc. (2) \textbf{41} (1990), no.~1, 51--62.
  \MR{1063542 (91j:20105a)}

\bibitem[Yeh82]{Yeh82}
S.~el~B. Yehia, \emph{Extensions of simple modules for the {C}hevalley groups
  and its parabolic subgroups}, Ph.D. thesis, University of Warwick, 1982.

\end{thebibliography}

\end{document}